\documentclass{article}
\usepackage{amsmath, amssymb}
\usepackage{amsthm}
\usepackage{paralist, geometry}
 \usepackage[colorlinks=true]{hyperref}
\hypersetup{urlcolor=blue, citecolor=red}
\usepackage{stmaryrd}
\usepackage{tikz}
\usepackage[utf8]{inputenc}

\newtheorem{theorem}{Theorem}[section]
\newtheorem{corollary}{Corollary}
\newtheorem*{main}{Main Theorem}
\newtheorem{lemma}[theorem]{Lemma}
\newtheorem{proposition}{Proposition}

\theoremstyle{definition}
\newtheorem{definition}[theorem]{Definition}
\newtheorem{remark}{Remark}

\newtheorem*{example}{Example}

\newcommand\A{\mathcal A}
\newcommand\F{\mathcal F}
\renewcommand\L{\mathcal L}
\newcommand\N{\mathbb N}
\newcommand\Z{\mathbb Z}
\newcommand\az{\A^\Z}
\newcommand\azd{\A^{\Z^d}}
\newcommand\s\sigma
\newcommand\h{h_{\mathrm{top}}}
\newcommand\Q{\mathbb Q}

\title{Effect of quantified irreducibility on\\ the computability of subshift entropy}

\author{Silvère Gangloff and Benjamin Hellouin de Menibus\footnote{The second author was supported by Basal PFB-03 CMM, Universidad de Chile, and did this work in part at in part in the Departamento de Matématicas, Universidad Andrés Bello, Republica 220, Santiago, Chile and Centro de Modelamiento Matematico, Beauchef 851, Santiago, Chile.}
}

\begin{document}
\maketitle

\centerline{\scshape Silvère Gangloff}
\medskip
{\footnotesize
 \centerline{Institut de Mathématiques de Toulouse,}
   \centerline{Université Paul Sabatier Toulouse 3,}
   \centerline{118 route de Narbonne, Toulouse, France.}
   \centerline{silvere.gangloff@math.univ-toulouse.fr}
} 

\medskip

\centerline{\scshape Benjamin Hellouin de Menibus$^*$}
\medskip
{\footnotesize
 \centerline{Laboratoire de Recherche en Informatique}
   \centerline{Université Paris-Sud - CNRS - CentraleSupelec, Université Paris-Saclay}
   \centerline{France}
   \centerline{hellouin@lri.fr}
}

\bigskip


\begin{abstract}
We study the difficulty of computing topological entropy of subshifts subjected to mixing restrictions. This problem is well-studied for multidimensional subshifts of finite type: there 	exists a threshold in the irreducibility rate where the difficulty jumps from computable to uncomputable, but its location is an open problem. In this paper, we establish the location of this threshold for a more general class, subshifts with decidable languages, in any dimension.

\end{abstract}

\section{Introduction}

Topological entropy is a real parameter widely used in the study of dynamical systems as a conjugacy invariant and a measure of dynamical complexity. The problem of effectively computing topological entropy -- that is to say, given a description of a dynamical system and $\varepsilon > 0$, computing its topological entropy with a maximum error of $\varepsilon$ -- has been considered for many systems. These efforts, leading to positive as well as negative answers, have been documented by Milnor in 2002 \cite{Milnor}.

For instance, entropy is not computable in general for cellular automata \cite{Hurd}, Turing machines \cite{DelvenneBlondel}, iterated piecewise affine maps on the interval $[0,1]$ \cite{Koiran}, smooth mappings in dimension $\ge 2$ and smooth diffeomorphisms in dimension $\geq 3$ \cite{Misiu}, etc. However, entropy is computable for positively expansive cellular automata \cite{Damico}, one-tape Turing machines \cite{Jeandel} and piecewise monotonic maps of the interval in some circumstances \cite{MilnorTresser}. More examples can be found in \cite{Milnor}. In many cases, these works characterized the class of real numbers that can appear as entropy of a system in the studied class (see e.g. \cite{HochmanMeyerovitch, Hochman, Guillon}).

The case of one dimensional subshifts of finite type (SFT) is well understood. Entropy is known to be computable through a simple method based on computing the largest eigenvalue of a graph associated with the subshift. Furthermore this method characterizes the numbers which are entropy of such a system by an algebraic condition: namely, they are exactly the non-negative rational multiples of logarithms of Perron numbers \cite{Lind}. 

The case of higher-dimensional SFT remained open for a long time, with many specific examples being studied and solved approximately or exactly using \emph{ad hoc} methods, especially by the statistical physics community; see \cite{Lieb, Engel, pavlov2} among many others.
A negative answer came much later in the seminal work of \cite{HochmanMeyerovitch}, where the authors proved that 
entropy of a multidimensional SFT is not computable in general, and that numbers realizable as entropy of a multidimensional SFT are characterized by a computability condition: all real numbers that are $\Pi_1$-computable (i.e. upper-semi-computable). 

In both settings, various authors studied the effect of dynamical restrictions, particularly mixing properties, on the difficulty of 
computing entropy. While the situation for mixing one-dimensional SFT is unchanged \cite{Lind}, entropy becomes computable for higher-dimensional SFT with strong mixing properties \cite{HochmanMeyerovitch}. For the particular case of two-dimensional SFT, Pavlov and Schraudner proved that entropy is even exptime-computable under a different type of mixing condition (block-gluing) \cite{Pavlov}, with a partial characterization. It seems natural in this context to introduce a notion of irreducibility rate that corresponds to the strength of the mixing restriction. We prove that entropy is computable when the irreducibility rate is below a certain level, but we are unable to locate the threshold marking the difficulty jump between computable and uncomputable cases.\bigskip

In this article, we consider subshifts that are not necessarily of finite type but that can be described by an algorithm in some sense (\textbf{decidable} subshifts), hoping that results on this class will provide insights for the finite type case. In general entropy of these subshifts is not computable \cite{Simonsen} and all $\Pi_1$-computable numbers can be realised as entropy \cite{Hertling}, but entropy becomes computable under strong mixing conditions \cite{Spandl}. This is very similar to the situation for multidimensional SFT.

In this more general context, we are able to characterize precisely the location of the threshold in the irreducibility rate marking the difficulty jump between the computable and uncomputable cases. More precisely, our new results are the following:

\begin{main}[Theorem~\ref{thm:EntropieCalculable}, Theorem~\ref{thm:EntropieCalculable2} and Theorem~\ref{thm:main}]~
Let $f : \N \rightarrow \N$ be a nondecreasing function.
\begin{enumerate}
\item  If $\sum_n \frac{f(n)}{n^2}$ converges at a computable rate, there exists an algorithm that computes the entropy of $f$-irreducible decidable subshifts.

\item If $\sum_n\frac{f(n)}{n^2} = +\infty$, the set of numbers that appear as the entropy of a decidable $f$-irreducible subshift are exactly the $\Pi_1$-computable numbers, and there exists no algorithm that computes the entropy of these systems.
\end{enumerate}
\end{main}

Missing definitions, in particular the notion of convergence of a series with computable rate, appear futher in this text. The first result applies, e.g., to any $O\left(\frac{n}{\log_2^{1+\alpha}}\right)$ function. Both results apply to one-dimensional and multidimensional decidable subshifts. The state of the art and new results are summed up in Table~1. \bigskip

\begin{table}[htp]
{\renewcommand{\arraystretch}{1.2}
  \begin{tabular}{|c||c|c|c|c|c|}
    \hline Subshift class & \multicolumn{4}{c|}{\begin{tabular}{@{}c@{}}Mixing properties\\\hspace{0.3cm}None \hspace{1.5cm}Weak\hspace{1.3cm} Strong \hspace{1cm} Very strong\end{tabular}}\\
    \hline
    \hline SFT & $\Pi_1$-comp. \cite{HochmanMeyerovitch}&?&computable $\dagger$& computable \cite{HochmanMeyerovitch} \\
    $d\geq 2$&all $\Pi_1$ reals \cite{HochmanMeyerovitch} &?&?&partial char. \cite{Pavlov}\\
    \hline Decidable& $\Pi_1$-comp. \cite{Simonsen} & $\Pi_1$-comp. $\dagger$& computable $\dagger$&computable \cite{Spandl}\\
    $d\geq 1$& all $\Pi_1$ reals \cite{Hertling} & all $\Pi_1$ reals $\dagger$&?&?\\
    \hline
  \end{tabular}
  }
  \caption{(First line) \label{table} Computational difficulty of computing the entropy; (Second line) Set of possible entropies. ``Weak'' and ``Strong'' mixing stand for irreducibility rates above or below the threshold, respectively; ``Very strong'' stands for constant irreducibility rates, or similar properties. ``$\Pi_1$-comp.'' means that the problem is $\Pi_1$-computable, but not computable; ``$\Pi_1$ reals'' stands for the set of $\Pi_1$-computable reals; $\dagger$ symbols indicate the contribution of the present article.}
\end{table}

Results in this text are of two kinds. On the one hand, we describe explicit algorithms to approximate the entropy in some cases, providing a computational upper bound on the difficulty of the problem. On the other hand, given a class of real numbers defined by their computational complexity, we build a family of subshifts whose entropy take all values in this class. This proves that computing the entropy is at least as hard as computing all numbers from this class, giving a computational lower bound.

\section{Definitions}

\subsection{Subshifts}

Let $\A$ be a finite set called alphabet. We call symbols the elements of the alphabet. 

Let $\mathbb{U} \subset \Z$ a finite subset of $\Z$. A \textbf{pattern} on the alphabet $\A$ and support $\mathbb{U}$ is some element of $\A^{\mathbb{U}}$. Denote $\A^\times$ the set of patterns on $\A$, that is, $\A^\times = \bigcup_{\substack{\mathbb{U}\subset \Z}}\A^{\mathbb{U}}$, 
the union is over the finite subsets $\mathbb{U}$. A pattern of support $[0,n-1]^d$ is an \textbf{$n$-block}. When $d=1$, we call $n$-block \textbf{words} and denote $\A^\ast$ the set of all words.

The set $\azd$ is the $d$-dimensional \textbf{full shift}. We often omit the dimension when a result is valid for all $d$. An element of $\azd$ is also called a \textbf{configuration}.

The \textbf{cylinder} associated to a pattern $u\in\A^\mathbb{U}$ and position $\vec{i}\in\Z^d$ is defined as
\[[u]_i = \{x\in\azd\ :\ x_{i+\mathbb{U}} = u\}\]

We say that a pattern $u\in\A^{\mathbb{U}}$ \textbf{appears} in a configuration $x\in\azd$ when $x \in [u]_{\vec{i}}$ for some $\vec{i} \in \Z^d$. Similarly, a pattern $v\in\A^{\mathbb{V}}$ appears in (or is a subpattern of) another pattern $u\in\A^{\mathbb{U}}$ when $V\subset \mathbb{U}$ and $u|_{\mathbb{V}} = v$.
For a pattern $u\in\A^{\mathbb{U}}$, and a symbol $a\in\A$, we denote the number of symbols $a$ appearing in the pattern $u$ as follows: 
\[\#_aw = \#\{i\in \mathbb{U}\ :\ w_i=a\}.\]

Denote $\vec{e}^1, \dots, \vec{e}^d$ the canonical set of generators of $\Z^d$. For $i \in\{1, \dots, d\}$, we call the $i$th \textbf{shift} function the function $\s_i: \azd\to\azd$ such that  
\[\s_i(x)_j = x_{j+e_i}\qquad \text{for all }x \in \az\text{ and }\vec{j} \in\Z^d.\]
 Shifts define an action of $\Z^d$ on the full shift denoted $\s$ (the shift action).

The $d$-dimensional full shift $\A^{\Z^d}$ endowed with the product of the discrete topology is a topological space. A \textbf{subshift} is a closed sub-space of the full shift which is stable under the action of the shift functions defined above. 

The \textbf{language} of a subshift $\Sigma$, denoted $\L(\Sigma)$, is the set of patterns that appear in some configuration of $\Sigma$. Formally, 
 \[\L_{\mathbb{U}}(\Sigma) = \{u\in\A^{\mathbb{U}}\, :\, \exists x\in\Sigma,\ x\in[u]_0\}\qquad \L(\Sigma) = \bigcup_{\substack{\mathbb U\subset \Z\\\text{finite}}}\L_{\mathbb{U}}(\Sigma).\]

Any subshift $\Sigma$ can be defined by a set of forbidden patterns, meaning there exists a set $\mathcal{F}\subset\A^\times$ such that 
\[\Sigma = \{x\in\azd\ :\ \forall u \in \mathcal{F},\ \forall \vec{j} \in \Z^d,\ x\notin[u]_{\vec{j}}.\}\]

For example, one can choose $\F_\Sigma = \L(\Sigma) ^c$, but this choice is not unique.
A subshift is of \textbf{finite type} (denoted SFT) when it can be defined by a finite set of forbidden patterns.

A pattern $w$ is said \textbf{locally admissible} for $\Sigma$ if no patterns of $\mathcal{F}$ appears in $w$. By opposition, we say $w$ is \textbf{globally admissible} if $w\in\L(\Sigma)$. Being locally admissible depends on the choice of $\mathcal{F}$, but we omit the set when some result applies to any choice of $\F$.

We denote $d$ the distance on $\Z^d$ defined for all $\vec{i},\vec{j}$ by $d(\vec{i},\vec{j}) = \max_{1\leq k\leq d} |\vec{i}_k-\vec{j}_k|$.

\subsection{Entropy}
 
Let $d \geq 1$, and $\Sigma$ a $d$-dimensional subshift.
The couple $(\Sigma,\s)$ is a dynamical system. 

The \textbf{entropy} of $\Sigma$ is the following number: 

\[\h(\Sigma) = \lim_{n\to\infty}\frac{\log_2(\#\L_n(\Sigma))}{n^d} = \inf_{n\to\infty}\frac{\log_2(\#\L_n(\Sigma))}{n^d}\]

The second equality is a well-known result. See \cite{LindMarcus}, Chapter 4, for a proof and more information about this notion.

In this formula, and in the remainder of the paper, the logarithm is in base two. \bigskip

For $0<\varepsilon<1$, the \textbf{binary entropy} of $\varepsilon$ is given by $H(\varepsilon) = -\varepsilon\log(\varepsilon)-(1-\varepsilon)\log(1-\varepsilon)$. Despite its name, it is not related to subshift entropy, but to the (information-theoretical) entropy of a Bernoulli process and happens to be useful in our proofs. Notice that $H(\varepsilon)\to 0$ as $\varepsilon \to 0$.

\subsection{Irreducibility with intensity function}

The \textbf{diameter} of a finite subset $\mathbb{U}$ of $\Z^d$ is $\delta(\mathbb{U}) = \max_{\substack{i\in \mathbb{U}\\j\in \mathbb{U}}}d(i,j)$, and the distance between two finite subsets $\mathbb{U},\mathbb{V}$ of $\Z^d$ is defined as $d(\mathbb{U},\mathbb{V}) = \min_{\substack{i\in \mathbb{U}\\j\in \mathbb{V}}}d(i,j)$. There is no ambiguity with the distance on $\Z^d$ since we use different notations for subsets and elements.

\begin{definition}[Irreducibility, block-gluing] \label{def.irr1}
Let $f : \N\to\N$ be a function, and $\Sigma$ a subshift.
$\Sigma$ is said to be \textbf{$f$-irreducible} when for all $\mathbb{U},\mathbb{V}$ finite subsets of $\Z^d$, all $u\in\L_{\mathbb{U}}(\Sigma)$, and all $v\in\L_{\mathbb{V}}(\Sigma)$ such that 
\[d(\mathbb{U}, \mathbb{V}) \ge f(\max(\delta(\mathbb{U}),\delta(\mathbb{V}))),\]
there exists a configuration $x \in \Sigma$ such that 
$x_{\mathbb{U}} = u$ and $x_{\mathbb{V}} =v$.

If the previous definition is true only when $\mathbb U$ and $\mathbb V$ are $n$-blocks, $\Sigma$ is said to be \textbf{$f$-block gluing}.
\end{definition}

\begin{remark} \label{remark.growing}
If a subshift is $f$-irreducible (resp. block-gluing), and $g \geq f$, then the subshift is also $g$-irreducible (resp. block-gluing).
\end{remark}

\begin{definition} \label{def.irr2}
  Let $f : \N\to\N$ be a function. A subshift $\Sigma$ is said to be $O(f)$-irreducible if there exists $g\in O(f)$ such that $\Sigma$ is $g$-irreducible.

\end{definition}

\begin{remark}
  This notion is related to and extends various mixing properties that appear in other texts.

  \begin{itemize}
  \item $\Sigma$ is strongly irreducible (also called the specification property) iff it is $O(1)$-irreducible;
  \item $\Sigma$ is topologically mixing iff there exists some $f$ such that $\Sigma$ is $f$-irreducible;
  \item $\Sigma$ is block-gluing iff it is $O(1)$-block-gluing.
  \end{itemize}

\end{remark}

\begin{definition} A subset $\mathbb{U}\subset\Z$ is
connected when:
\[\forall i\le j \le k \in\mathbb{U},\ i,k \in \mathbb{U} \Rightarrow j \in \mathbb{U}.\]

For a subset $\mathbb U\subset Z$, we denote $\gamma(\mathbb{U})$ the smallest connected subset containing $\mathbb U$ (in other words, its convex hull).
\end{definition}
\begin{proposition} \label{prop:equivalencegluing}
  Let $f:\N\to\N$. A $\Z$-subshift is $f$-block gluing if and only if it is $f$-irreducible.
\end{proposition}
\begin{proof}
It is clear that $f$-irreducibility implies $f$-block gluing.
We prove the converse.
First, let us prove some property of couples of subsets of $\Z$. \bigskip

Let $\mathbb{U}, \mathbb{V}$ be two non-empty subsets of $\Z$
such that $d(\mathbb{U}, \mathbb{V}) > 0$.
Let $(\mathbb U_i)_{1\leq i\leq k}$ be the maximal connected components of $\mathbb U$ that satisfy $\gamma(\mathbb U_i)\cap \mathbb V = \emptyset$, and define $(\mathbb V_i)_{1\leq i\leq l}$ similarly. Fixing $\mathbb U'_i = \gamma(\mathbb U_i)$, $\mathbb U' = \bigcup_i\mathbb U'_i$, and $\mathbb V'_i$ and $\mathbb V'$ similarly, it is easy to check that:
\begin{enumerate}
\item $\mathbb{U} \subset \mathbb{U}'$ and $\mathbb{V} \subset \mathbb{V}'$,
\item $\delta (\mathbb{U}', \mathbb{V}') = \delta (\mathbb{U}, \mathbb{V})$,
\item $|l-k|\leq 1$ and one of the following statements is true:
\begin{itemize}
  \item $\forall i,\ \mathbb{U}'_i \le \mathbb{V}'_{i} \le \mathbb{U}'_{i+1}$ (when defined)
  \item $\forall i,\ \mathbb{V}'_i \le \mathbb{U}'_{i} \le \mathbb{V}'_{i+1}$ (when defined).
\end{itemize}
\end{enumerate}

Let $f:\N\to\N$ and $\Sigma$ be some $f$-block gluing one-dimensional subshift.
Consider $\mathbb{U},\mathbb{V}$ subsets of $\Z$ such that  
\[d(\mathbb{U},\mathbb{V}) \ge f(\max(\delta(\mathbb{U}),
\delta(\mathbb{V}))),\]
and $u\in\L_\mathbb{U}(\Sigma), v\in\L_\mathbb{V}(\Sigma)$. Since $u$ and $v$ are globally admissible there exist $u'\in\L_{\mathbb U'}(\Sigma), v\in\L_{\mathbb V'}(\Sigma)$ such that $u'|_U=u$
and $v'|_U=v$. Without loss of generality assume $l=k-1$ and that $\mathbb{U}'_i \le \mathbb{V}'_{i} \le \mathbb{U}'_{i+1}$ for all $i$.


Denote $u^i = u'|_{\mathbb U_i}$, $v^i = v'|_{\mathbb V_i}$ and $n= \max \left( \delta(\mathbb{U}),\delta(\mathbb{V})\right)$.
Then $d (\mathbb{U}'_1, \mathbb{V}'_1) \ge f(n)$. Since $\mathbb{U}'_1\subset \mathbb U'$ and $\mathbb{V}'_1\subset \mathbb V'$, we have $n \ge \max (\delta (\mathbb{U}'_1) , \delta (\mathbb{V}'_1))$. Since $\Sigma$ is $f$-block gluing, there exists $w\in\L_{\gamma (\mathbb{U}'_1 \cup \mathbb{V}'_1)}$ such that $w|_{\mathbb{U}'_1} = u^1$ and $w|_{\mathbb{V}'_1} = v^1$.\bigskip

Iterating the same argument, we have  $n\geq \max(\delta (\gamma (\mathbb{U}'_1 \cup \mathbb{V}'_1)), \delta (\mathbb{U}'_2))$ and $d(\gamma (\mathbb{U}'_1 \cup \mathbb{V}'_1), \mathbb{U}'_2)\geq f(n)$. By $f$-block gluing, there exists $w'\in\L_{\gamma ( \mathbb{U}'_1 \cup \mathbb{V}'_1 \cup \mathbb{U}'_2)}$ such that $w'|_{\mathbb{U}'_1} = u^1$, $w'|_{\mathbb{V}'_1} = v^1$ and $w'|_{\mathbb{U}'_2} = u^2$.\bigskip

Through an easy induction, we see that there exists $w\in\L_{\gamma (\mathbb{U}'\cup\mathbb{V}')}$ such that $w|_{\mathbb U'}=u'$ and $w|_{\mathbb V'}=v'$.
\end{proof}

\subsection{Computability}

Turing machines are a model of computation that is believed to capture the informal notion of computation (Church-Turing thesis). Although it is not needed in this text, we give the formal definition of a Turing machine for completeness.

A Turing machine consists in a head attached with a symbol representing its internal state, 
which is able to move over and read/write on an infinite tape filled with symbols, and to change its internal state each 
time it goes over a symbol according to a transition function.

Formally, a Turing machine is a tuple $(\mathcal{A},\mathcal{Q},q_0,q_h,\#,\delta)$
where $\mathcal{A}$ is a finite set (the tape alphabet), 
$\mathcal{Q}$ is a finite set (the states alphabet), $q_0 \in \mathcal{Q}$ 
is the initial state, and $q_h$ the halting state, $\#\in\A$ is a special blank tape symbol and $\delta: \mathcal{A} \times \mathcal{Q} \rightarrow \mathcal{A} \times \mathcal{Q} \times \{\rightarrow,\leftarrow\}$ is called the transition function.
When a machine is in state $q$ and reads a symbol $a$ at its current position, $\delta(q,a)$ specifies the new tape symbol being written, the new state of the head and the direction in which the head moves at this step.

Initially, the tape contains a finite word from $\A^\ast$ (the input) surrounded by blank symbols, the head is on position $0$ with state $q_0$. The evolution of the machine is then determined by the rules described above. If the machine enters the halting state $q_h$ (the machine halts), the output is the finite non-blank portion of the tape at this point.\bigskip

A function $\A^\ast \to \A^\ast$ is said to be \textbf{computable} if there exists a Turing machine that, given as input $u\in\A^\ast$, eventually halts and outputs $f(u)$. This notion extends to functions $\N\to\N$ through a binary encoding, and to other countable input and output sets by appropriate encodings. We present a notion of computability for real numbers and subshifts as needed in this text.

\paragraph*{Real numbers:} 

\begin{definition} \label{def.comp.real}
A real number $\alpha$ is said to be
\begin{itemize}
\item {\bf{computable}} if there exists a computable function $\varphi_\alpha : \N \to \Q$ such that for all $n$,
\[|\alpha-\varphi_\alpha(n)|\leq 2^{-n}.\]
\item {\bf{$\Pi_1$-computable}} (or upper-semi-computable) 
if there exists a computable function $\varphi_\alpha : \N \to \Q$ such that for all $n$ 
\[\alpha = \inf_{n\in\N} \varphi_\alpha(n).\]
\end{itemize}
\end{definition}

\begin{remark}
In the definition for a $\Pi_1$-computable number given above, we also have 
\[\alpha = \inf_{n \in \N} \varphi_\alpha'(n)\]
where $\varphi_\alpha'(n)=\inf_{k \leq n} \varphi_\alpha(k)$ for all $n$. $\varphi_\alpha'$ is also computable and is nonincreasing, which means that the function $\varphi_{\alpha}$ in the definition can be chosen nonincreasing.
\end{remark} \bigskip

\begin{remark}
A computable real number is in particular $\Pi_1$-computable. However, there exist $\Pi_1$-computable numbers 
that are not computable. For instance, take the real number 
\[\sum_{k=0}^{+\infty} \varepsilon_k 2^{-k},\] 
where $\varepsilon_k=0$ if the $k$th Turing machine stops on the empty input and $\varepsilon_k=1$ otherwise. This number is $\Pi_1$-computable but not computable: otherwise, the halting problem would be decidable.
\end{remark}

\paragraph*{Subshifts:} 

General subshifts do not admit a description by a finite sequence of symbols, so we need to restrict to a countable subclass: decidable subshifts, defined below.

\begin{definition}
  A subshift $\Sigma$ is {\bf{decidable}} if $\L(\Sigma)$ is decidable; that is, the function
  \[\begin{array}{ccl}\A^\times &\to& \{0,1\}\\
    u&\mapsto& \left\{\begin{array}{ll} 1 & \text{if }u\in\L(\Sigma)\\0&\text{otherwise,}\end{array}\right.
    \end{array}\]
is a computable function.  
\end{definition}

\subsection{Irreducibility and decidability for subshifts of finite type}

It is a well-known fact that multidimensional SFT ($d\geq 2$) are not decidable in general; this is sometimes referred to as the undecidability of the extension problem. However, some dynamical properties (minimality, block gluing, etc.) make subshifts of finite type decidable.

\begin{proposition}\label{prop:MixSFT}
Let $\Sigma$ a $d$-dimensional subshift of finite type on some alphabet $\A$.
If it is $f$-irreducible with $f$ a computable function 
such that $f(n)=o(n)$, then it is decidable.
\end{proposition}

The remainder of this section is dedicated to a proof of this statement. It is a slight generalization from the proof of Corollary 3.5 in \cite{HochmanMeyerovitch}, which proved the same result for $O(1)$-irreducible subshifts. \bigskip

Let $\Sigma$ a $d$-dimensional SFT for some $d \ge 1$ defined by a set of forbidden patterns $\F$. We assume without loss of generality that elements of $\F$ are $r$-blocks for some $r>0$ (by replacing every pattern $u$ in $\F$ by the set of $r$-blocks that contains it, assuming $r\geq \delta(u)$).\bigskip

We denote $C_n = \llbracket -n,n \rrbracket ^d$
and  $D_{m} = C_{m+r} \backslash C_n.$ for all $n < m$ integers.

\begin{lemma} \label{lem.irr1}
  Let $u \in\A^{C_n}$. $u\notin\L(\Sigma)$ if, and only if, there exists some $N>n$ such that for every $p^N\in\A^{C_N}$ locally admissible for $\Sigma$, $p^N|_{C_n} \neq u$.
\end{lemma}

In other words, a pattern $u$ is not globally admissible if and only if there is a larger window such that the pattern cannot be extended to this window without creating a forbidden pattern.

\begin{proof}
Let $u \in \A^{C_n}$.
\begin{description}
\item[$(\Rightarrow)$] If $u$ appears in some configuration $x\in\Sigma$, $p^N = x|_{C_N}$ verifies the second assertion. 
\item[$(\Leftarrow)$] The reciprocal uses a compacity argument. Assume that $u$ verifies the second assertion. Fix $\gamma^0=u$. Since $\A_{C_{n+1}}$ is a finite set, extract a subsequence $(p^{N_1(k)})_k$ from $(p^N)_N$ such that $p^{N_1(k)}|_{C_{n+1}}$ is constant; denote this value $\gamma^1$. Now iterate this process, extracting for all $i$ a subsequence $(p^{N_{i+1}(k)})_k$ from $(p^{N_i(k)})_k$ such that $p^{N_{i+1}(k)}|_{C_{n+i+1}}$ is constant, denoted $\gamma^{i+1}$.

  We constructed a sequence of locally admissible patterns $\gamma^k \in \A^{C_{n+k}}$ such that $\gamma^0=u$ and $\gamma^{k+1}|_{C_{n+k}} = \gamma_k$ for all $k$. Now define a configuration $x$ where $x_{\vec{i}}$ is the common value of $\gamma^k_{\vec{i}}$ for all $k$ large enough. Since all $\gamma_k$ are locally admissible, it follows that $x\in\Sigma$. Because $u$ appears in $x$, it follows that $u\in\L(\Sigma)$.
\end{description} \end{proof}

\begin{lemma} \label{lem.irr2}
Assume that $\Sigma$ is $o(n)$-irreducible.
A pattern $u \in\A^{C_n}$ is globally admissible for $\Sigma$ if and only if 
there exists some integers $N_0$ and $m$ such that 
\begin{itemize}
\item $N_0\ge m \ge n$;
\item for all $N\ge N_0$ and every $v\in \A^{D_{m}}$ such that $w|_{D_m}=v$ for some locally admissible pattern $w\in\A^{C_N}$, there exists a locally admissible pattern $w' \in \A^{C_N}$ such that $w'|_{C_n} = u$ and $w'|_{D_m} = v$.
\end{itemize}
\end{lemma}

\begin{proof}
  If $u$ verifies the last condition, it appears in particular inside some pattern on $C_N$ for an arbitrarily large $N$, so it is globally admissible by the same argument as in the proof of Lemma~\ref{lem.irr1}.

Let us prove the direct implication. Take $f$ a computable function such that $f(k)=o(k)$ and $\Sigma$ is $f$-irreducible, and assume that $u \in \L(\Sigma)$.

For all $k \ge n$, $C_n$ and $D_k$ have both 
diameter at most $2k+2r$ and $d(C_n,D_k) = k-n$. Let $m$ be the smallest integer such that $m-n \ge f(2m+2r)$. 
From the $f$-irreducibility property of $\Sigma$, we have 
that for any globally admissible $v\in\A^{D_{m}}$, there exists a configuration $x \in \Sigma$ such that $x|_{D_{m}} = v$ and $x|_{C_n}=u$. 

By Lemma~\ref{lem.irr1}, for every pattern $v \in \A^{D_{m}}$ that is not globally admissible, there exists some $N_v$ such that for all $N \ge N_v$, there is no locally admissible pattern $w\in C_N$ such that $w|_{D_{m}}=v$. Fix $N_0 = \max_{v \in \A^{D_{m}}} N_v$. It follows that for all $N \ge N_0$, all the locally admissible patterns $w\in\A^{C_{N}}$ satisfy $w|_{D_{m}} \in \L(\Sigma)$.

Consider any locally admissible pattern $w \in \A^{C_N}$. As noted above, $w|_{D_{m}} \in \L(\Sigma)$, so by $f$-irreducibility there exists $x\in\L_{C_{m+r}}(\Sigma)$ such that $x|_{C_n} = u$ and $x|_{D_{m}} = w_{D_{m}}$. Now define $w'$ by putting $w'|_{C_{m+r}}=x$ and $w'|_{C_N \backslash C_{m+r}}=w|_{C_N \backslash C_{m+r}}$. We prove that $w'$ is locally admissible.

Remember forbidden patterns of $\Sigma$ are $r$-blocks. A translated $r$-block is either entirely included in $C_{m+r}$ or in ${C_N \backslash C_{m}}$. But $w'|_{C_{m+r}} = x$ and $w'|_{{C_N \backslash C_{m}}} = w|_{{C_N \backslash C_{m}}}$, and both being locally admissible, this ends the proof.

This process is illustrated in Figure~\ref{fig.lem.irr2}. 
\begin{center}
\begin{figure}[htp]
\centering
\begin{tikzpicture}[scale=0.25]
\begin{scope}
\fill[gray!40] (0,0) rectangle (10,10);
\draw (0,0) rectangle (10,10);
\node at (5,-1) {$w\in\A^{C_N}$ (loc. adm.)};
\draw[dashed] (2,2) rectangle (8,8); 
\draw[dashed] (3,3) rectangle (7,7);
\node at (2.5,2.5) {$v$};

\draw[-latex] (12,5) -- (15,5);
\end{scope}
\begin{scope}[xshift=15cm]
\fill[gray!40] (2,2) rectangle (8,8);
\fill[gray!20] (3,3) rectangle (7,7);
\fill[gray!40] (4,4) rectangle (6,6);

\draw (2,2) rectangle (8,8);
\draw (3,3) rectangle (7,7);
\draw (4,4) rectangle (6,6);
\node at (5,5) {$u$};
\node at (2.5,2.5) {$v$};
\draw[-latex] (11,5) -- (14,5);
\end{scope}

\begin{scope}[xshift=30cm]
\fill[gray!40] (0,0) rectangle (10,10);
\fill[gray!20] (3,3) rectangle (7,7);
\fill[gray!40] (4,4) rectangle (6,6);
\draw[dashed] (2,2) rectangle (8,8);
\draw (3,3) rectangle (7,7);
\draw (4,4) rectangle (6,6);
\draw (0,0) rectangle (10,10) ;
\node at (5,5) {$u$};
\node at (2.5,2.5){$v$};
\node at (5,-1) {$w'\in\A^{C_N}$ (loc. adm.)};
\end{scope}

\end{tikzpicture}
\caption{Every pattern $v$ on $D_m$ appearing in some locally admissible pattern of $\A^{C_N}$ appears jointly with $u$ in some other locally admissible pattern.}
\label{fig.lem.irr2}
\end{figure}
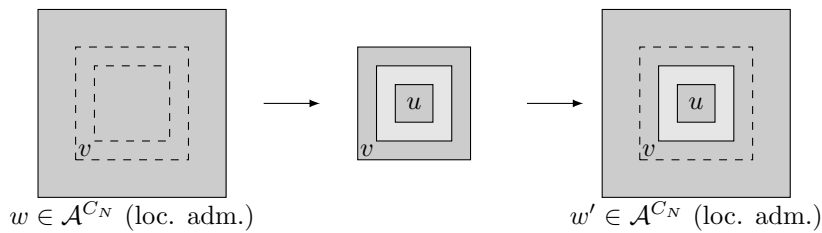
\end{center}
\end{proof}

Using the previous lemmas, we finish the proof of Proposition~\ref{prop:MixSFT}.\bigskip

Take as input a pattern $u$ and assume without loss of generality that $u\in\A^{C_n}$. Do the following in parallel:
  \begin{enumerate}
  \item For every integer $N$, check whether Lemma~\ref{lem.irr1} applies; if this is the case, output $u\notin\L(\Sigma)$.
  \item For every integers $m$ and $N_0$, check whether Lemma~\ref{lem.irr2} applies; if this is the case, output $u\in\L(\Sigma)$.
    \end{enumerate}
  
By Lemma~\ref{lem.irr1} and Lemma~\ref{lem.irr2}, exactly one of these processes will stop and output the correct answer. This means that $\Sigma$ is decidable.

\section{Computability of the entropy of subshifts}

From now on we consider decidable subshifts. 
This section is organized as follows: 
\begin{itemize}
\item In Section~\ref{sec.upper.bound} we present the state of the art for computing entropy in SFT and decidable subshifts.
\item in Section~\ref{sec.comp} we prove that when $f$ satisfies some summability condition, then the entropy of $f$-irreducible decidable subshifts is computable.
\item In Section~\ref{sec.uncomp}, we prove that when $f$ does not satisfy this condition, every $\Pi_1$-computable number can be realised as the entropy of an $f$-irreducible decidable subshift, which implies that the entropy of this class of subshifts is uncomputable in general.
\end{itemize}

Hence this summability condition defines some kind of threshold that delimits the computable and uncomputable cases.

\subsection{\label{sec.upper.bound} State of the art}

In this section, we present known results that are either folklore or appeared in the litterature.

\begin{definition} \label{def.algo}
  Let $d \ge 1$ and $\mathcal{C}$ a class of $d$-dimensional decidable subshifts. We say there is an algorithm that computes the entropy of $\mathcal C$ when there is a Turing machine taking as input some integers $(n,m)$ and outputs some rational number $r_{n,m}$ such that:
\[\left|\h(\Sigma) - r_{n,m} \right|\leq 2^{-m}.\]
where $\Sigma$ is the decidable subshift defined by the $n$th Turing machine.\bigskip

We say that the algorithm upper semi-computes the entropy of $\mathcal C$ when instead
\[ \h (\Sigma) = \inf_m r_{n,m}.\]
\end{definition}

\begin{remark}
  In Definition~\ref{def.algo}, if the number given as input to the algorithm does not correspond to a Turing machine that decides the language of a $d$-dimensional decidable subshift, the behaviour of the algorithm is unspecified.

  Intuitively, there is a uniform way to compute the entropy for this class of subshifts, but the validity of the input cannot be checked.
\end{remark}

\begin{proposition}\label{Prop:ObstructionGenerale}
There is an algorithm that upper-semi-computes the entropy of
decidable $d$-dimensional subshifts. In particular, entropies of decidable $d$-dimensional subshifts are $\Pi_1$-computable real numbers.
\end{proposition}

This result also holds for $d$-dimensional SFT. The following result states that this upper bound is tight.

\begin{proposition}[\cite{Hertling}, Theorem 22]\label{Prop:RealisationGenerale}
Let $\alpha$ be a $\Pi_1$-computable number. 
Then there exists a decidable subshift $\Sigma$ such that $\h(\Sigma) = \alpha$.
\end{proposition}

A similar result for $d$-dimensional SFT, $d>1$, was obtained in \cite{HochmanMeyerovitch}. This proposition can also be obtained as a corollary of Theorem~\ref{thm:main}.

The situation improves when considering subshifts with mixing properties:

\begin{theorem}[Theorem 1.3 in \cite{HochmanMeyerovitch}] The entropy of a $O(1)$-irreducible $d$-dimensional SFT is a computable real number.
\end{theorem}

\begin{remark}
This phenomenon was already observed in \cite{Spandl} (Theorem 6.11).
\end{remark}

This last result suggests that irreducibility relates to the computability of the entropy in SFT. It is natural to ask for which rates of irreducibility the entropy of SFT is computable or uncomputable. We consider in the rest of this paper a larger class of subshifts, decidable subshifts, and determine a threshold that delimits the irreducibility rates for which the entropy is computable of uncomputable.

\subsection{\label{sec.comp} Under the threshold}

In this section, we prove that when the irreducibility rate $f$ satisfies some summability condition, the entropy of the class of $f$-irreducible decidable subshifts is computable. These results apply to $f$-irreducible SFT as well. 

\begin{definition} Let $(a_n)_n$ be a sequence of 
non-negative numbers. The series $\sum a_n$ is said to converge at a computable rate when there is a computable function $n : \N \rightarrow \N$ (the rate) such that for all $t \in \N$, 
\[\left| \sum_{n \ge n(t)} a_n \right| \le 2^{-t}.\]
\end{definition}

\begin{lemma} \label{lemma.comp} Let $(a_n)_n$ and $(b_n)_n$ two series of non-negative integers, 
such that for all $n$, $a_n \le b_n$. If the series $\sum b_n$ converges at a computable rate, 
then the series $\sum a_n$ also converges at computable rate.
\end{lemma}

\begin{theorem}\label{thm:EntropieCalculable}
Let $f$ be a non-decreasing computable function such that the series $\sum \frac{f(n)}{n^2}$ converges at a computable rate. 
Then there is an algorithm that computes the entropy of $f$-block gluing one-dimensional decidable subshifts.
\end{theorem}

\begin{remark}
In particular, values of the entropy for subshifts in this class are computable real numbers. In this theorem, $f$-block gluing can be replaced by $f$-irreducible.
\end{remark}
 
\begin{remark}
  In the statement of Theorem~\ref{thm:EntropieCalculable}, the algorithm depends on $f$. However, a careful reading of the proof shows that there exists an algorithm which takes as input positive integers $c$ and $m$ and, assuming that $m$ represents a Turing machine that decides the language of some $c$-block gluing $\Z$-subshift, outputs the entropy of this subshift.
\end{remark}

\begin{proof} Let $m$ be some integer such that the $m$th Turing machine decides the language of a one-dimensional $f$-block gluing subshift $\Sigma \subset \az$.
By definition of the $f$-block gluing property, for any two words $u,v\in\L_n(\Sigma)$, 
there exists a word $w\in\A^{f(n)}$ such that $uwv\in\L_{2n+f(n)}(\Sigma)$. Because $u,v$ 
can be chosen freely we have 
 \[\#\L_n^2(\Sigma) \leq \#\L_{2n+f(n)}(\Sigma) \leq |\A|^{f(n)}\cdot \#\L_{2n}(\Sigma),\]
the second inequality coming from the fact that any globally admissible word of length $2n+f(n)$ can be decomposed into a globally admissible word of length $2n$ and some word of length $f(n)$ for a crude upper bound. By taking the logarithm of this inequality and dividing by 
$2n$, we get: 

\[\frac{\log_2 \#\L_n(\Sigma)}{n} \leq \frac{\log_2 \#\L_{2n}(\Sigma)}{2n}
+ \log_2 (|\A|) \frac {f(n)}{2n}\]

The previous inequality being true from any $n$, we apply it iteratively on 
the sequence $(2^{n+k})_{k \ge 0}$, and combining the $l \ge 1$ first inequalities 
we get: 
\begin{equation}\frac{\log_2 \#\L_{2^n}(\Sigma)}{n} \leq \frac{\log_2 \#\L_{2^{n+l}}(\Sigma)}{2^{n+l}}
+ \log_2 (|\A|) \sum_{k=1}^{l} \frac {f(2^{n+k})}{2^{n+k}}\label{eq:entropyapprox}\end{equation}

Let us study the rightmost series. Since $f$ is non-decreasing, we have that for all $k$,
  \[\frac{f(2^k)}{2^{k}} \leq \frac{1}{2^{2k}}\sum_{i=2^k}^{2^{k+1}-1}f(i) \leq \frac{1}{4} 
	\sum_{i=2^k}^{2^{k+1}-1} \frac{f(i)}{i^2}\]
        The first inequality comes from the fact that $f(2^k)\leq f(i)$ for $i\geq 2^k$, and the second inequality we use that $i \leq 2^{k+1}$.

Using Lemma~\ref{lemma.comp}, and the theorem hypothesis, the 
series $\sum_{k\geq n}\frac{f(2^k)}{2^{k}}$ converges at a computable rate. \bigskip

Since $\sum_{k\geq n}\frac{f(2^k)}{2^{k}}$ converges at a computable rate, there exists a computable function  
$t \mapsto n(t)$ such that for all $t$, 
\[ \sum_{k=n(t)+1}^\infty \frac{f(2^{k})}{2^{k}} \le \frac{2^{-t-1}}{\log_2 (|\A|)}\]
This implies that 
\[\left|\h(\Sigma)-\frac{\log_2\#\L_{2^{n(t)}}(\Sigma)}{2^{n(t)}}\right|\leq 2^{-t-1}.\]

To conclude, the algorithm to approximate $\h(\Sigma)$ up to a precision $2^{-t}$ 
with input $(m,t)$ runs as follows:
\begin{enumerate} 
\item compute $n(t)$, 
\item then count all words of $\L_{2^n(t)}(\Sigma)$ (this is possible because $\Sigma$ is decidable), 
\item then compute a rational approximation of $\frac{\log_2\#\L_{2^n(t)}(\Sigma)}{2^{n(t)}}$, 
using an approximation of the logarithm function up to a precision $2^{-t-1}$, 
which is in turn a rational approximation of $\h(\Sigma)$ up to precision $2^{-t}$.
\end{enumerate}

We conclude that the entropy $\h(\Sigma)$ is uniformly computable.
\end{proof}

We now extend this proof to the case of $d$-dimensional subshifts, for $d \ge 2$.

\begin{theorem}\label{thm:EntropieCalculable2}
Let $d \ge 1$ and $f$ be a non-decreasing computable function such the series $\sum\frac{f(n)}{n^2}$ converges at a computable rate. Then there is an algorithm that computes the entropy of $f$-block gluing $d$-dimensional decidable subshifts.
\end{theorem}

\begin{corollary}\label{cor:EntropieCalculable2}
  For a function $f$ that verifies the same conditions as Theorem~\ref{thm:EntropieCalculable2},
  \begin{enumerate}
    \item there is an algorithm that computes the entropy of $f$-irreducible $d$-dimensional decidable subshifts.
    \item if furthermore $f(n)=o(n)$, there is an algorithm that computes the entropy of $f$-block gluing (or $f$-irreducible) $d$-dimensional subshifts of finite type.
      \end{enumerate}
\end{corollary}

\begin{proof}
Let $d \ge 1$, and $m$ an integer representing a Turing machine that decides the language of some $d$-dimensional decidable subshift $\Sigma$ which is $f$-block gluing. Denote for $n_1, \dots, n_k \in \N$,
\[C[n_1,\dots, n_d]  = \llbracket 0,n_1-1 \rrbracket \times ... \times \llbracket 0,n_d-1 \rrbracket.\]
The diameter of this set is $N = \max_k n_k$. \bigskip

Fix $\mathbb U=C_{n_1,\dots, n_d}$ and $\mathbb V = \mathbb U+\vec{i}$ where $\vec{i} = (n_k+f(N))\vec{e}^k$ for some $k$. From the $f$-block gluing property of $\Sigma$, for all $u\in\L_\mathbb{U}(\Sigma)$ and $v\in\L_\mathbb{V}(\Sigma)$, there exists a pattern $w\in\L_\mathbb{W}(\Sigma)$ where $\mathbb W=C[n_1, n_2 ... n_{k-1}, 2n_k +f(N),n_{k+1}, ... n_d]$ such that $w|_\mathbb{U}=u$ and $w|_\mathbb{V}=v$.\bigskip

Therefore, taking $k=1$: 
\[\begin{array}{ccl}
\#\L_{C[n_1,\dots, n_d]}^2 & \leq & \#\L_{C[2n_1+f(N),n_2,\dots n_d]}\\
& \leq & \#\L_{C[2n_1,n_2,\dots, n_d]} \cdot |\A|^{f(N)\times n_2 \times ... 
\times n_d} \\
& \leq & \#\L_{C[2n_1, n_2 ,\dots, n_d]} \cdot |\A|^{N^{d-1} f(N)},\\
\end{array}\]

where the first inequality is by the above remark, the second inequality comes from decomposing $w$ into two patterns on $C[2n_1, n_2 , .. , n_d]$ and on $C[f(N), n_2 , .. , n_d]$, respectively, and the last is by definition of $N$.
\bigskip

We have $\delta(C[2n_1,n_2, ... , n_d])\leq 2N$, so applying the same argument: 
\[\#\L_{C[2n_1,n_2,\dots, n_d]}^2 \leq 
\#\L_{C[2n_1,2n_2, n_3 ,\dots, n_d]} \cdot |\A|^{(2N)^{d-1} f(2N)}.\]
Iterating this argument, we have by an easy induction
\begin{align*}\#\L_{C[n_1,\dots, n_d]}^{2^d} &\leq 
\#\L_{C[2n_1,\dots, 2n_d]} \cdot (|\A|^{(2N)^{d-1}f(2N)})^{d-1} ) \cdot 
|\A|^{N^{d-1} f(N)}\\&\leq (|\A|^{(2N)^{d-1}f(2N)})^{d}\end{align*}
where we used the fact that $f$ is nondecreasing. Applying this equation to $n_1 = ... = n_d=n$ for some $n$,
\begin{align*}\#\L_{C[n,\dots, n]}^{2^d} \leq& \#\L_{C[2n,\dots, 2n]} \cdot 
(|\A|^{(2n)^{d-1}f(2n)})^{d}\\
\frac{\log_2(\#\L_{C[n,\dots, n]})}{n^d} 
\leq& \frac{\log_2(\#\L_{C[2n,\dots, 2n]})}{(2n)^d} 
+ d\cdot\log_2 (|\A|)\cdot\frac{f(2n)}{2n}.\end{align*}

The end of the proof is similar to the proof of Theorem~\ref{thm:EntropieCalculable}.
\end{proof}

  \begin{example} \label{cor.2}
For any $0<\varepsilon<1$, there exists an algorithm that computes the entropy of $n/(\log_2(n+c))^{1+\varepsilon}$-block gluing decidable subshifts for any constant $c > 2^{1+\varepsilon}$.
\end{example}

$c$ is only needed so that the function is nonincreasing, which can be checked by a straightforward computation.
    
  \begin{proof}Let $\varepsilon>0$ and $f : n\mapsto n/(\log_2(n))^{1+\varepsilon}$. 
By the remark that follows Definition~\ref{def.irr1}, any $f$-irreducible subshift is also $g$-irreducible, with $g : n \mapsto n/(\log_2(n))^{1+r}$, for some rational number $r<\varepsilon$.

We now prove that the series
\[\sum_k \frac{g(k)}{k^2} = \sum_{k} \frac{1}{k\log_2(k)^{1+r}}\]
converges at a computable convergence rate, so that the same result follows for $f$. We use the following inequality: for all $n$, 
\[\sum_{k \geq n} \frac{1}{k\log_2(k)^{1+r}} \leq \int_{n-1}^{\infty} 
\frac{dt}{t\log_2(t)^{1+r}} = \frac{1}{r\log_2(n-1)^r},\]
using the fact that $t \mapsto 1/t\log_2(t)^{1+r}$ is non-decreasing.

Thus the series converges with rate \[n : t \mapsto \left\lceil \exp\left(\left(\frac {2^t}r\right)^{1/r}\right)+1\right\rceil,\]
which is a computable function. Indeed, with this function 
  \[\sum_{k \ge n(t)} \frac{g(k)}{k^2}\le\frac{1}{r\log_2(n(t)-1)^r} \le 2^{-t}\]
We conclude using Theorem~\ref{thm:EntropieCalculable}.
\end{proof}

\subsection{\label{sec.uncomp} Above the threshold}

In this section, we consider decidable subshifts whose irreducibility rate is above the threshold, and prove the following theorem: 

\begin{theorem}\label{thm:main}
  Let $\alpha > 0$ be a $\Pi_1$-computable real number and $f : \N \to\N$ be a computable non-decreasing function such that 
\[\sum_{n=1}^\infty \frac{f(n)}{n^2} = +\infty.\]

Then there exists a decidable $f$-irreducible one dimensional subshift with entropy $\alpha$.
\end{theorem}

First we prove this theorem for $d=1$ using \textbf{bounded density subshifts}, introduced by Stanley~\cite{Stanley} (as bounded density \textit{shifts}). We introduce these objects in Section~\ref{sec.bounded.density.shifts} and study their properties in the rest of the section, the actual proof consisting in a construction presented in Section~\ref{sec.algo}. 
The result is then extended to any $d \ge 2$ in Section~\ref{sec.dimsup}.

In all this section, $f$ is a function which verifies the conditions of Theorem~\ref{thm:main}, and we denote $F : n \mapsto 2n+f(n)$.

\subsubsection{\label{sec.bounded.density.shifts} Bounded density subshifts}

\begin{definition}[Bounded density subshift] \label{def.bounded.density.shift}
Let $p = (p_n)_{n\ge 1}$ be a sequence of positive integers. The bounded density subshift associated to $p$, denoted $\Sigma_p\subset \{0,1\}^{\Z}$, is the one-dimensional subshift 
defined by the set of forbidden words ${\mathcal{F}}^p = \bigcup_{n \ge 1} {\mathcal{F}}^p_n$, where 
for all $n$:
\[{\mathcal{F}}^p_n = \left\{ u\in {\mathcal{A}}^{n}\ :\ \#_1u> p_n \right\}.\]
\end{definition}

\begin{example}

\begin{enumerate}
\item The bounded density subshift $\Sigma_p$ defined by the sequence 
$p =(1,1,2,3,4,...)$ is the Golden mean shift $\Sigma\subset\{0,1\}^\Z$ defined by the set of forbidden words $\mathcal{F} = \{11\}$.

  Indeed,
  \begin{itemize}
  \item $11\notin\L_2(\Sigma_p)$ since $\#_111>p_2$. Therefore $\Sigma_p\subset \Sigma$.
  \item For $n>1$, any word $w\in\L_n(\Sigma)$ must contain at least a zero (or it would contain $11$), so that $\#_1w\leq p_n$ and $w$ is not a forbidden pattern for $\Sigma_p$. Therefore $\Sigma\subset\Sigma_p$.
  \end{itemize}

Equivalently, the sequence $p' = (1,1,2,2,3,3,4,4, ...)$ defines the same subshift. Indeed,
\begin{itemize}
  \item because $p' \le p$, we have $\Sigma_{p'} \subset \Sigma_p$. 
  \item Reciprocally, take $w \in \L_n(\Sigma_p)$. Since $11$ is forbidden, every $1$ in $w$ is followed by a zero, except if it is the last symbol. It follows that $2\#_1w-1\leq n$, and therefore $\#_1w\leq \lceil n/2 \rceil = p'_n$, which means that $w\in\L_n(\Sigma_{p'})$. We conclude that $\Sigma_{p} \subset \Sigma_{p'}$.
\end{itemize}

\item For $n\in\N$, define $\Sigma_n$ the bounded density subshift associated to the sequence $p^n$ defined as follows: 
\begin{itemize}
\item if the $n$-th Turing machine starting on an empty input stops before computing $k$ steps, then $p^n_k = p_{k-1}$;
\item otherwise, $p^n_k = p_{k-1}+1$. 
\end{itemize}
If the $n$-th Turing machine never stops on the empty input, then $p_n = n$ for all $n$ and $\Sigma_n = \{0,1\}^\Z$ the full shift, with entropy $1$. On the other hand, if the $n$-th Turing machine stops at some point, then $p^n$ is ultimately 
constant and $\Sigma_n$ has entropy zero by Lemma~\ref{lem:slopetozero}.
\end{enumerate}
\end{example}

\subsubsection{Bounded density subshifts are decidable}

To use bounded density subshifts in the proof of Theorem~\ref{thm:main}, we prove sufficient conditions for them to be decidable in Lemma~\ref{lem:freqDecidable}. Lemma~\ref{lem:extendbyzero} is used in the proof of Lemma~\ref{lem:freqDecidable}, and is used again throughout the section.

\begin{lemma} \label{lem:extendbyzero}
Let $p$ be a non-decreasing sequence of integers, $\F_p$ the set of forbidden patterns defined in Definition~\ref{def.bounded.density.shift}, and $w$ a word.

If $w\notin \F_p$, then $0^m w 0^n\notin \F_p$ for all integers $m,n$.

If $w$ is locally admissible for $\Sigma_p$, then $0^m w 0^n$ is globally admissible for $\Sigma_p$ for all integers $n,m$.
\end{lemma}

\begin{proof}
The first statement follows from the fact that $\#_10^nw0^m = \#_1w$.

For the second statement, the element $x$ of $\{0,1\}^{\Z}$ such that 
\begin{itemize}
\item $x_{\llbracket 0, |w|-1 \rrbracket} = w$
\item and for all $i \in \Z \backslash \llbracket 0, |w|-1 \rrbracket$, $x_{i} = 0$. 
\end{itemize}
belongs to $\Sigma_p$. 
Indeed, any finite word appearing in $x$ is amongst the following types: 
\begin{enumerate}
\item $0^k u$, where $k \ge 0$ and $u$ is some prefix of $w$,
\item $v 0 ^k$, where $k \ge 0$ and $v$ is some suffix of $w$,
\item $0^m w 0^n$, where $n, m \ge 0$.
\end{enumerate}
In each case, $u, v$ and $w$ cannot be forbidden patterns since $w$ is locally admissible. By the first statement, none of these words can be a forbidden pattern. As a consequence, for all $n,m \ge 0$, $0^m w 0^n$ is globally admissible.
\end{proof}

\begin{lemma}\label{lem:freqDecidable}
If $p$ is non-decreasing and computable, then $\Sigma_p$ is decidable.
\end{lemma}

\begin{proof} 
Since $p$ is non-decreasing, Lemma~\ref{lem:extendbyzero} implies any locally admissible pattern of $\Sigma_p$ is globally admissible. Hence the algorithm to decide if an input word $w\in\{0,1\}^l$ is in $\mathcal{L}(\Sigma_p)$ runs as follows:
  \begin{itemize}
  \item compute the value of $p_k$ for all $k\leq l$;
  \item compute the set of forbidden patterns of length $\leq l$;
  \item check if one of these patterns appears in $w$.
	\item if this is the case, return $0$ ; else, 
	return $1$. \qedhere
\end{itemize}\end{proof} 

\subsubsection{Zero limit density implies zero entropy}

In this section, we prove a technical lemma related 
to the entropy of bounded density subshifts. 
This lemma will be used in the proof 
of Theorem~\ref{thm:main} to control the entropy of the constructed subshift.

For $p$ a sequence of positive integers, 
the \textbf{limit density} of $p$ is the number $\lim \frac{p_n}{n}$ when 
it exists.

\begin{lemma}\label{lem:slopetozero}
  Let $p$ be some sequence such that $\frac{p_n}{n} \rightarrow 0$.
Then $\h(\Sigma_p)=0$.
	
In other words, if the limit density is zero, then the entropy is also zero.

This result holds under the weaker condition $\inf_n \frac{p_n}{n} = 0$.
\end{lemma}

\begin{proof}
  Consider a sequence $p$ such that $\inf_n \frac{p_n}{n} = 0$. Since for all $n$, any word of $\L_n(\Sigma_p)$ has less than $p_n$ symbols equal to $1$, we have:
\[\#\L_n(\Sigma_p)\leq \sum_{k=0}^{p_n} \binom{n}{k}.\]
Take some $n$ such that $\frac{p_n}n\leq \frac 12$. We have:
\[\sum_{k=0}^{p_n} \binom{n}{k} \leq (p_n+1) \binom{n}{p_n},\]
since $k \mapsto \binom{n}{k}$ is non-increasing on $\Bigl \llbracket 0, \Bigl \lfloor \frac{n}{2} \Bigr \rfloor \Bigr \rrbracket$. Therefore:

\[h(\Sigma_p)\leq \frac{\log\#\L_n(\Sigma)}{n}\leq\frac{\log(p_n+1)}{n} + \frac{1}{n} \log \left( 
\binom{n}{p_n} \right).\]

To estimate $\frac 1n\log \left(\binom{n}{p_n} \right)$, we use a known computation trick:

\[\binom{n}{p_n} \left(\frac {p_n}n\right)^{p_n}\left(1-\frac {p_n}n\right)^{n-{p_n}} \leq \sum_{k=0}^n{\binom{n}{k}}\left(\frac kn\right)^{k}\left(1-\frac kn\right)^{n-k} = 1.\]
Applying the logarithm function to this inequality leads to:
\[\frac{1}{n} \log \left(\binom{n}{p_n} \right) \leq -\frac {p_n}n\log\left(\frac {p_n}n\right)-\left(1-\frac {p_n}n\right)\log\left(1-\frac {p_n}n\right) = H\left(\frac{p_n}n\right),\]
This is true for all $n$, and $H(\varepsilon)\to 0$ when $\varepsilon\to 0$.
It follows that if $p_n /n$ tends towards $0$, or if $\inf_n \frac{p_n}{n} = 0$,we have $\h(\Sigma_p)=0$. 
\end{proof}

\subsubsection{Piecewise linear density functions}

In the proof of Theorem~\ref{thm:main}, we use the bounded density subshifts given by a piecewise linear density sequence $p$, as defined in Definition~\ref{def.betaf.shifts}. Remember that $F : n\mapsto 2n+f(n)$.

\begin{definition}[Piecewise linear density functions] \label{def.betaf.shifts}
Given $f : \N\to\N$ a nondecreasing function and $\beta = (\beta_k)_{k \ge 0}$ a sequence of positive rational numbers, we define a function $\varphi (f,\beta) : \mathbb{R}_{+} \rightarrow \mathbb{R}_{+}$ that is linear on intervals $[F^n(1), F^{n+1}(1)]$: for all $n \ge 1$ and $t \in [0,F^{n+1}(1)-F^n(1)]$, 
\[\varphi(f,\beta) (F^{n} (1) + t) = \varphi (f,\beta) ( F^n (1)) + t.\beta_n.\]
Then, we define $p(f,\beta) : \N \rightarrow \N$ by 
\[p(f,\beta)_n = \Bigl \lceil \varphi (f,\beta) (n)\Bigr \rceil.\]

If $\beta$ is a finite sequence, we define a corresponding infinite sequence $\beta'$ as follows:  
  \[\beta'_k = \left\{\begin{array}{ll}\beta_k&\text{if }k\leq n\\\beta_n &\text {if } k>n\end{array}\right.\]
By abuse of notation, we denote $\varphi (f,\beta) = \varphi (f, \beta')$ and $p(f,\beta)= p(f,\beta')$. The bounded density subshift obtained from a sequence $p(f,\beta)$, where $\beta$ can be finite or infinite, is denoted $\Sigma_{f,\beta}$.
\end{definition}

Figure~\ref{fig.def.betaf.shifts} illustrate schematically 
this definition. 

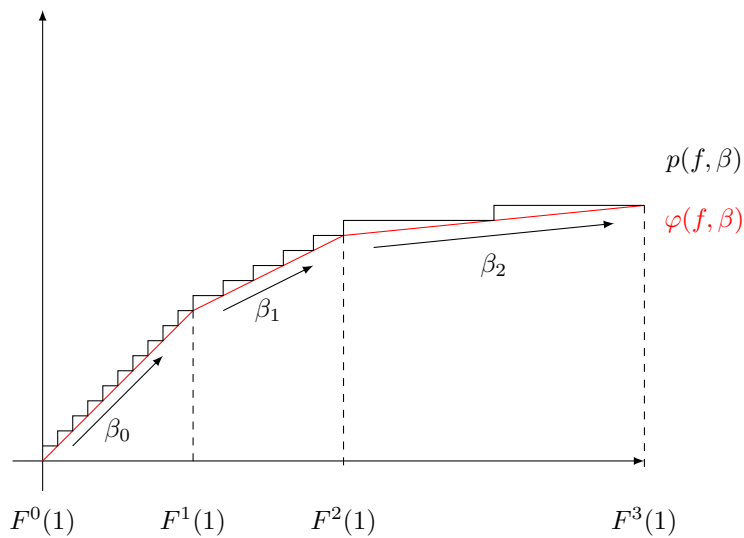
\begin{figure}[htp]
\centering
\begin{tikzpicture}[scale=0.4]

\draw[-latex] (0,-1) -- (0,15); 
\draw[-latex] (-1,0) -- (20,0); 
\draw[color=red] (0,0) -- (5,5);
\draw[color=red] (5,5) -- (10,7.5);
\draw[color=red] (10,7.5) -- (20,8.5);

\draw[dashed] (5,5) -- (5,-0.25);
\draw[dashed] (10,7.5) -- (10,-0.25);
\draw[dashed] (20,8.5) -- (20,-0.25);
\node at (0,-2) {$F^0 (1)$};
\node at (5,-2) {$F^{1} (1)$};
\node at (10,-2) {$F^{2} (1)$};
\node at (20,-2) {$F^{3} (1)$};

\draw[-latex] (1,0.5) -- (4,3.5); 
\node at (2.5,1) {$\beta_0$};
\draw[-latex] (6,5) -- (9,6.5);
\node at (7.5,5) {$\beta_1$}; 

\draw[-latex] (11,7.1) -- (19,7.9);
\node at (15,6.5) {$\beta_2$};

\draw (0,0.5) -- (0.5,0.5) -- (0.5,1) -- 
(1,1) -- (1,1.5) -- (1.5,1.5) -- (1.5,2) 
-- (2,2) -- (2,2.5) -- (2.5,2.5) -- (2.5,3) -- 
(3,3) -- (3,3.5) -- (3.5,3.5) -- (3.5,4) -- (4,4) 
-- (4,4.5) -- (4.5,4.5) -- (4.5,5) -- (5,5) ;

\draw (5,5) -- (5,5.5) -- (6,5.5) -- (6,6) -- (7,6) 
-- (7,6.5) -- (8,6.5) -- (8,7) -- (9,7) -- (9,7.5) -- (10,7.5) 
-- (10,8) -- (15,8) -- (15,8.5) -- (20,8.5);

\node at (22,10) {$p(f,\beta)$};
\node[color=red] at (22,8) {$\varphi(f,\beta)$};
\end{tikzpicture}
\caption{\label{fig.def.betaf.shifts} Illustration of Definition~\ref{def.betaf.shifts}.}
\end{figure}

\begin{lemma} \label{lem:beta.zero.density} Let $\beta$ be a non-increasing sequence of positive rational numbers such that $\beta_n \rightarrow 0$.
Then the entropy of $\Sigma_{f,\beta}$ is zero.
\end{lemma}

\begin{proof}
Since $\beta$ is non-increasing, for all $n \ge 1$ and $t \ge 0$, 
\[\varphi(f,\beta)(F^n (1) +t) \le \varphi(f,\beta)(F^n(1)) + t.\beta_n,\]
and as a consequence, 
\[p(f,\beta)_{F^n(1)+t} \le  \varphi(f,\beta)(F^n(1)) + t.\beta_n +1, \]
Hence 
\[\frac{p(f,\beta)_{F^n(1)+t}}{F^n(1)+t} \le \frac{\varphi(f,\beta)(F^n(1)) + t.\beta_n +1}{
F^n(1)+t} \underset{t\to\infty}\longrightarrow \beta_n.\]
This means that 
\[\limsup_k \frac{p_k}{k} \le \beta_n.\]
This is true for all $n$ and $\beta_n \rightarrow 0$, so $\frac{p_n}{n} \rightarrow 0$. By Lemma~\ref{lem:slopetozero}, we get $\h(\Sigma_{f,\beta})=0$.
\end{proof}

\begin{lemma}
Let $\beta$ be a computable sequence of rational numbers. Then the subshift $\Sigma_{f,\beta}$ is decidable.
\end{lemma}

\begin{proof}
This is an application of Lemma~\ref{lem:freqDecidable}. The conditions of this lemma are verified if $\beta$ is computable, 
since then $p(f,\beta)$ is computable, and always non-decreasing by definition.
\end{proof}

\subsubsection{Concaveness and irreducibility}

In the construction used in the proof of Theorem~\ref{thm:main}, we need to prove that the obtained subshift is $f$-irreducible. We show that this is the case when the sequence $p$ is concave, which is the case for a piecewise linear sequence $p(f,\beta)$ defined by a nonincreasing sequence $\beta$, and under a particular condition on some values of the sequence.

\begin{definition}
A sequence $p : \N \rightarrow \N$ is said to be \textbf{concave} when there exists a concave function $\varphi : \mathbb{R}_{+} \rightarrow \mathbb{R}_{+}$ such that for all $n \ge 1$, 
\[p_n = \Bigl \lceil \varphi (n) \Bigr \rceil.\]
\end{definition}

\begin{lemma} \label{lem:slopesinequality}
Let $p$ be a concave sequence of integers. For all integers $k,l,n \ge 0$, 
\[p_{n+k+l} - p_{n+l} \le p_{n+k} - p_n + 2.\]
\end{lemma}

\begin{proof}
Let $\varphi : \mathbb{R}_{+} \rightarrow \mathbb{R}_{+}$ such that for all $n \ge 1$, 
\[p_n = \Bigl \lceil \varphi (n) \Bigr \rceil.\]
Then for all $k,l \ge 0$ and $n \ge 1$, 

\begin{align*}
p_{n+k+l} + p_n - p_{n+l} - p_{n+k} 
 &=  \Bigl \lceil \varphi (n+k+l) \Bigr \rceil + \Bigl \lceil \varphi (n) \Bigr \rceil - \Bigl \lceil \varphi (n+l) \Bigr \rceil
- \Bigl \lceil \varphi (n+k) \Bigr \rceil \\
 &\le  \varphi (n+k+l) + \varphi (n) - \varphi (n+k) - \varphi (n+l) +2 \\ 
 &\le  2,
\end{align*}
since the function $\varphi$ is concave.
\end{proof}

\begin{lemma}\label{lem:abovemixing}
  Assume that $p$ is a concave sequence and that for 	all $n \ge 1$, $p_{F(n)}\geq 2p_{n}+4$. Then $\Sigma_p$ is $f$-irreducible.
\end{lemma}

\begin{proof}
Let $p$ be some sequence which verifies these conditions. We prove that $\Sigma_p$ is $f$-block gluing, which implies 
(Proposition~\ref{prop:equivalencegluing}) that $\Sigma_p$ is $f$-irreducible.

Consider some integer $n \ge 1$ and two words $u, v$ in $\mathcal{L} (\Sigma_p)$ with $|u|\leq n$, $|v|\leq n$. We prove that $u0^{f(n)}v \in \L(\Sigma_p)$. By Lemma~\ref{lem:extendbyzero}, 
it is sufficient to prove that this word is locally admissible.

Let $w$ be a subword of $u0^{f(n)}v$. It is amongst the following types: 
\begin{itemize}
\item $w=u' 0^{k}$, where $u'$ is a suffix of $u$ and $k \le f(n)$, 
\item $w=0^k v'$, where $v'$ is a prefix of $v$ and $k \le f(n)$, 
\item  $w=u' 0^{f(n)} v'$, where $u'$ is a suffix of $u$ and $v'$ is a prefix of $v$. 
\end{itemize}

In the first two cases, $w\in\mathcal{L} (\Sigma_p)$ by Lemma~\ref{lem:extendbyzero}.

In the third case, since $u'$ and $v'$ are in $\mathcal{L} (\Sigma_p)$, we have $\#_1 u' \leq p_{|u'|}$
and $\#_1 v' \leq p_{|v'|}$. 
Therefore $\#_1 w \leq p_{|u'|}+p_{|v'|}$. 

Since $p$ is a concave sequence, as a consequence of Lemma~\ref{lem:slopesinequality}, we have:
\begin{itemize}
\item $p_n - p_{|u'|} + 2 \geq p_{2n+f(n)}-p_{n+|u'|+f(n)}$, and
\item $p_n - p_{|v'|} + 2 \geq p_{n+|u'|+f(n)}-p_{|v'|+|u'|+f(n)}$.
\end{itemize}

Moreover, using the Lemma hypothesis, we have $p_{2n+f(n)}\geq 2p_n+4.$

Combining the above inequalities, we obtain 
\[\#_1 w = p_{|u'|}+p_{|v'|} \le p_{|v'|+|u'|+f(n)} = p_{|w|},\]
 
and therefore $w$ is not a forbidden pattern. This means that $u0^{f(n)}v$ is locally admissible for $\Sigma_p$, and globally admissible as a consequence of Lemma~\ref{lem:extendbyzero}.
\end{proof}

\subsubsection{Computability of the entropy of bounded density subshifts associated to a finite sequence}

In the construction, we will need to compute an approximation of the entropy of the subshift obtained after defining only the $n$ first terms of the sequence. This is the object of the following Lemma.

\begin{lemma}\label{lem:BoundedDecidable}
Let $f$ be a computable integer function. There exists an algorithm that, given as input a 
non-increasing finite sequence $(\beta_n)_{n\leq N}$ of positive rational numbers, computes the entropy of $\Sigma_{f,\beta}$.
\end{lemma}

\begin{proof}

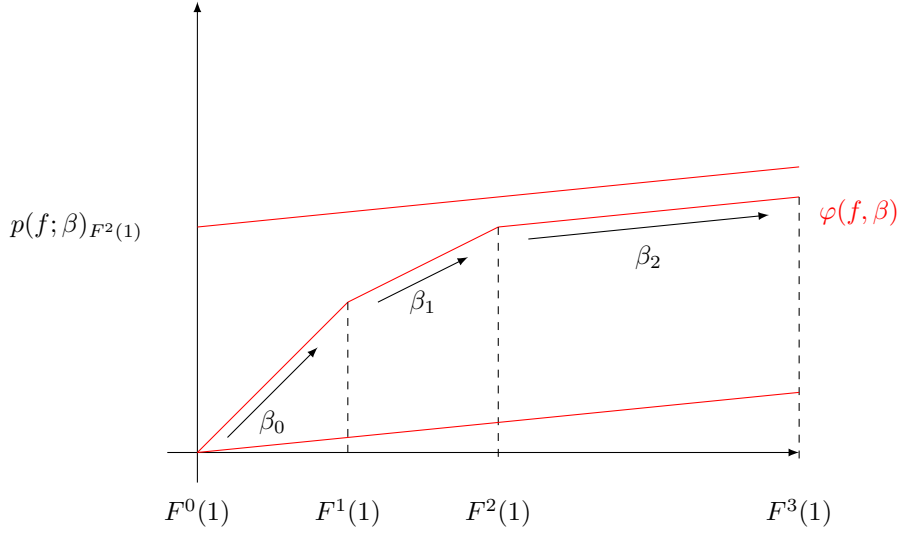
\begin{figure}[htp]
\centering
\begin{tikzpicture}[scale=0.4]

\draw[-latex] (0,-1) -- (0,15); 
\draw[-latex] (-1,0) -- (20,0); 
\draw[color=red] (0,0) -- (5,5);
\draw[color=red] (5,5) -- (10,7.5);
\draw[color=red] (10,7.5) -- (20,8.5);

\draw[dashed] (5,5) -- (5,-0.25);
\draw[dashed] (10,7.5) -- (10,-0.25);
\draw[dashed] (20,8.5) -- (20,-0.25);
\node at (0,-2) {$F^0 (1)$};
\node at (5,-2) {$F^{1} (1)$};
\node at (10,-2) {$F^{2} (1)$};
\node at (20,-2) {$F^{3} (1)$};

\draw[-latex] (1,0.5) -- (4,3.5); 
\node at (2.5,1) {$\beta_0$};
\draw[-latex] (6,5) -- (9,6.5);
\node at (7.5,5) {$\beta_1$}; 

\draw[color=red] (0,0) -- (20,2);
\draw[color=red] (0,7.5) -- (20,9.5);

\draw[-latex] (11,7.1) -- (19,7.9);
\node at (15,6.5) {$\beta_2$};
\node at (-4,7.5) {$p(f;\beta)_{F^2 (1)}$};

\node[color=red] at (22,8) {$\varphi(f,\beta)$};
\end{tikzpicture}
\caption{\label{fig:encadrementpentes} Illustration of the proof of Lemma~\ref{lem:BoundedDecidable}.}
\end{figure}

Let $\beta$ be a non-increasing sequence of non-negative rational numbers. Since $\beta$ is non-increasing, $\varphi(f,\beta)$ 
is concave and therefore $p(f,\beta)$ is concave. 
Since $\beta$ is non-increasing, 
\[\forall k\geq 1,\  p(f,\beta)_k \leq p(f,\beta)_{F^N (1)} + \Bigl \lceil \beta_N \cdot k \Bigr \rceil,\]
as illustrated schematically on Figure~\ref{fig:encadrementpentes}.

Now fix $c = \frac{F^N(1)+6}{\beta_N}.$ We have 
\begin{align*}p_{2n+c} &\geq p(f,\beta)_n+\beta_N(n+c)-1\\ &\geq 2p(f,\beta)_n - F^N(1) + \beta_Nc-2\\ &\geq 2p(f,\beta)_n+4,\end{align*}

where the first inequality uses the concavity of $p(f,\beta)$ and the second inequality uses the previous remark. By Proposition~\ref{lem:abovemixing}, $\Sigma_p$ is $c$-irreducible, and $c$ is computable from the knowledge of $\beta$ and $f$.

$p(f,\beta)$ is also computable since $f$ and $\beta'$ are computable.
It follows, by  Lemma~\ref{lem:freqDecidable}, that $\Sigma_{f,\beta}$ is decidable, and one can construct the Turing machine which decides its language from the knowledge of $\beta$.

The Lemma follows by application of Theorem~\ref{thm:EntropieCalculable2}.   
\end{proof}


\subsubsection{Controlling the change of entropy}

In this section, we give an upper bound on the decrease of entropy when defining the $(N+1)$ th term of the sequence $\beta$ when the $N$ first terms are already defined. In the construction, these conditions 
will be verified asymptotically.

\begin{lemma}\label{lem:boundentropy}
  Let $(\beta_n)_{n \leq N}$ be a finite non-increasing sequence of positive rational numbers, and $f$ a computable integer function. Denote $(\beta'_n)_{n\leq N}$ the sequence defined by 
$\beta'_N = \beta_N-\frac{1}{F^N(1)}$ and $\beta'_n=\beta_n$ for $n<N$.
  
	Then we have the following inequality: 
  \[\h(\Sigma_{f,\beta'}) \leq \h(\Sigma_{f,\beta}) \leq \h(\Sigma_{f,\beta'})
	+H\left(\frac 1{F^N(1)}\right),\]
	where $H$ is the binary entropy. 
	
\end{lemma}

\begin{proof}
The left-hand inequality comes from $\Sigma' \subset \Sigma$, 
since $\beta$ is non-increasing.

Let us prove the right-hand inequality. 
We define a family of functions ${\delta}_N : \A^\ast\to\A^\ast$ such that $\delta_N(\L_n(\Sigma))\subset \L_n(\Sigma')$ for all $n$, and evaluate the size of the pre-images of each word.

\begin{description}

\item[A remark on $p(f,\beta)$ and $p(f,\beta')$]

  We prove that $\forall n, p(f,\beta)_n\geq p(f,\beta')_n - \left\lfloor\frac n{F^N(1)}\right\rfloor$. Indeed,
  \begin{itemize}
  \item when $n\leq F^N(1)$, $p(f,\beta)_n= p(f,\beta')_n$;
  \item when $n > F^N(1)$,
    \begin{align*}p(f,\beta)_n-p(f,\beta')_n & \leq \lceil \varphi(f,\beta)_n\rceil-\lceil\varphi(f,\beta')_n\rceil\\
&\leq \left\lfloor \varphi(f,\beta)_n-\varphi(f,\beta')_n\right\rfloor + 1 \\&\leq \left\lfloor(\beta_{N} - \beta'_N)(n-F^N(1))\right\rfloor +1 \\&\leq \left\lfloor\frac n{F^N(1)}\right\rfloor,\end{align*}

     using the Lemma hypothesis, and where $\varphi$ is defined in Definition~\ref{def.betaf.shifts}.
    \end{itemize}
  
 \item[Definition of the function ${\delta}_N$]
   Take $w\in\A^n$ and define inductively:
   \begin{align*}
   \ell^N_0(w) &= \min\{k : w_k = 1\}\\\forall 0<j<\left\lfloor\frac n{F^N(1)}\right\rfloor,\ \ell^N_{j}(w) &= \min\{k > \ell^N_{j-1} : k>j F^N(1) \text{ and } w_k = 1 \}.\end{align*}

 $\ell^N_j(w)$ is undefined when the corresponding set is empty. Now define:
   \[
   \forall k,\ \delta_N (w)_{k} =\left \{
   \begin{array}{ll}0 & \text{if }\exists j,\ k=\ell^N_j(w)\\
                    w_k & \text{otherwise}.\end{array}\right.
   \]
   
Intuitively, the function marks every $F^N(1)$-th letter of the word. Going from left to right, it turns the first symbol 1 it meets after each mark. Our goal is that most words of length $n$ lose $\frac1{F^N(1)}$ symbols 1 that are well-distributed along the word. This definition is illustrated in Figure~\ref{fig:decoupagesuppression}.\bigskip

\begin{figure}[htp]
\centering
\begin{tikzpicture}[scale=0.4]

\foreach[count=\x] \n in {0,0,1,0,1,1,0,1,0, 0,0,0,0,0,0,0,0,0,0,1,0,1,1,0,1,1,1,0,0,0,0}
{\node at (.6*\x, 4.5) {$\n$};}

\foreach[count=\x] \n in {0,0,0,0,1,1,0,1,0, 0,0,0,0,0,0,0,0,0,0,0,0,0,1,0,1,1,1,0,0,0,0}
{\node at (.6*\x, -.5) {$\n$};}

\foreach \x in {3, 20, 22}
{\draw[-latex] (.6*\x, 4) -- (.6*\x, .3);
\draw[thick] (.6*\x-.3, 4.1) rectangle ++ (0.6, 0.8);
\draw[thick] (.6*\x-.3, -0.9) rectangle ++ (0.6, 0.8);}

\draw[line width = 0.4mm] (0,-0.5) -- (0,0.5) (21.6,-0.5) -- (21.6,0.5); 
\draw[line width = 0.4mm] (0,4.5) -- (0,5.5) (21.6,4.5) -- (21.6,5.5); 

\foreach \x in {1, ..., 3}
{\draw (.3+5.4*\x,-0.5) -- (.3+5.4*\x,0.5);
\draw (.3+5.4*\x,4.5) -- (.3+5.4*\x,5.5);} 

\draw[->] (0,5) to[bend left] node[midway, above] {$\ell^N_1(w)$} (1.8,5) ; 
\draw[->] (5.7,5) to[bend left]  node[midway, above] {$\ell^N_2(w)$} (12,5); 
\draw[->] (11.1,5) to[bend left]  node[midway, above] {$\ell^N_3(w)$}(13.2,5); 
\draw[->, dotted] (16.5,5) to[bend left]  node[midway, above] {$\ell^N_4(w)$} node[at end, right] {undefined}(19,6); 

\draw[-latex] (-2.5,4) -- (-2.5,0) node[above, at start]{$w$} node[left, midway]{$\delta_N$} node[below, at end] {${\delta}_N(w)$};

\draw[latex-latex] (0,7.5) -- (5.7,7.5) node[midway, above] {$F^N(1)$};
\end{tikzpicture}

\caption{\label{fig:decoupagesuppression} 
Illustration of the definition of the function ${\delta}_N$.}
\end{figure}
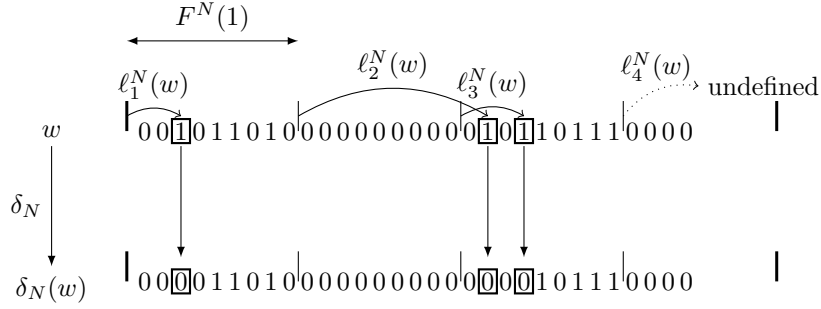
\item[The image of $\mathcal{L}_n (\Sigma)$ is in $\mathcal{L}_n (\Sigma')$ for all $n$]

Take $w \in \L_n(\Sigma)$. We prove that ${\delta}_N (w) \in \L_n (\Sigma')$. Using Lemma~\ref{lem:extendbyzero}, it is sufficient to prove that this word is locally admissible for $\Sigma'$.

By the same argument as Lemma~\ref{lem:extendbyzero}, any fobidden pattern for $\Sigma'$ contains another forbidden pattern ending by a $1$. Therefore we take two arbitrary coordinates $i, i+m-1$ such that $0\leq i<i+m-1\leq n$, assuming that $\delta_N(w)_{i+m-1} = 1$, and we prove that $\delta_N(w)_{[i,i+m-1]}$ is not a forbidden pattern for $\Sigma'$.

Now take any $j$ such that $i\leq jF^N(1) \leq i+m-1$. Since $\delta_N(w)_{i+m-1} = 1$, we have a fortiori $w_{i+m-1} = 1$, which means that $\ell^N_j(w)$ is well-defined and $i\leq \ell^N_j(w) \leq i+m-1$. Since this holds for $\left \lfloor \frac m{F^N(1)} \right \rfloor$ different values of $j$, we have $\#_1(\delta_N(w)_{[i,i+m-1]}) \leq \#_1(w_{[i,i+m-1]}) - \left \lfloor \frac m{F^N(1)} \right \rfloor$ by definition of $\delta_N$.

Since we took $w\in\L_n(\Sigma)$, we have $\#_1(w_{[i,i+m-1]}) \leq p(f,\beta)_m$, and therefore \[\#_1(\delta_N(w)_{[i,i+m-1]}) \leq p(f,\beta)_m - \left \lfloor \frac m{F^N(1)} \right \rfloor \leq p(f,\beta')_m,\] using the first part of the current proof. This means that $\delta_N(w)_{[i,i+m-1]}$ is not a forbidden pattern for $\Sigma'$. Since this is true for all $i$ and $m$, $w$ is locally admissible, and therefore globally admissible, for $\Sigma'$.

\item[Entropy inequality]
  Take any $w\in \L_n (\Sigma')$. From the definition of ${\delta}_N$, a preimage of $w$ by $\delta_N$ must be equal to $w$ except for at most $\left \lfloor \frac n{F^N(1)} \right\rfloor$ coordinates. Therefore $w$ has at most $\binom{n}{\lfloor n/F^N(1) \rfloor}$ different preimages.

  Since we proved $\delta_N(\L_n(\Sigma))\subset \L_n(\Sigma')$ for all $n$, it follows that: \begin{align}\#\L_n(\Sigma') &\leq \binom{n}{\lfloor n/F^N(1) \rfloor}\#\L_n(\Sigma)\nonumber\\
    \frac{\log \#\L_n(\Sigma')}n &\leq \frac{\log \binom{n}{\lfloor n/F^N(1) \rfloor}+\log \#\L_n(\Sigma))}n\label{eq:entropydecrease}\end{align}
  
    By Stirling's formula, $\log(n!) \sim n\log n$. It follows:

    \begin{align*}\log \binom{n}{\lfloor \frac n{F^N(1)} \rfloor} &\sim n\log n - \frac n{F^N(1)}\log \left (\frac n{F^N(1)} \right) - \left(n-\frac n{F^N(1)} \right)\log\left(n-\frac n{F^N(1)}\right)\\
  \begin{split} &\sim -\frac n{F^N(1)}\left (\log \left(\frac n{F^N(1)}\right) - \log n\right ) \\
  &\hspace{3.5cm} - \left(n-\frac n{F^N(1)}\right)\left (\log\left(n-\frac n{F^N(1)} \right)-\log n\right)\end{split}\\
      &\sim n \cdot \left(-\frac 1{F^N(1)} \log\left (\frac 1{F^N(1)}\right) - \left(1-\frac 1{F^N(1)}\right)\log\left(1-\frac 1{F^N(1)}\right)\right)\\
      &\sim nH\left(\frac 1{F^N(1)}\right).
      \end{align*}

Coming back to Equation~(\ref{eq:entropydecrease}), we see that $\h(\Sigma')\leq \h(\Sigma)+H\left(\frac 1{F^N(1)}\right)$, which is the desired statement.

\end{description}
\end{proof}

\subsubsection{Sketch of the proof 
of Theorem~\ref{thm:main}}

The idea of the proof for Theorem~\ref{thm:main} is as follows. Given a $\Pi_1$-computable number $\alpha>0$ and a computable non-decreasing function $f : \N \to\N$, we build an algorithm that computes a non-increasing sequence of rational numbers $(\beta)$ such that:
\begin{itemize}
\item $\Sigma_{f,\beta}$ has entropy $\alpha$;
\item $p(f,\beta)$ satisfies the conditions of Lemma~\ref{lem:abovemixing} are verified (ensuring the $f$-irreducibility).
\end{itemize}
The decidability is ensured by the computability of $\beta$.

Since we only know approximate values $\alpha_n$ of $\alpha$ with no known rate of convergence, we build intermediate subshifts with entropies approximately $\alpha_n$, and use the summability condition on $f$ to ensure that the final subshift has entropy $\alpha$.

\subsubsection{ \label{sec.algo0} 
Description of the algorithm}

Let $\alpha>0$ be a $\Pi_1$-computable real number and $(\alpha_n)_n$ a non-increasing computable sequence of rational numbers such that $\alpha_n \searrow \alpha$.
We define inductively a sequence $(\beta^{\alpha}_n)_{n\in\N}$ such that the associated bounded density subshift $\Sigma_{f,\beta^{\alpha}}$, denoted $\Sigma_{\alpha}$, satisfies $\h(\Sigma_{\alpha}) = \alpha$; this induction can be seen as a recursive algorithm.

Take $n \in \N$ and define $\beta_n$ as follows:

\begin{itemize}
\item $\beta^{\alpha}_0 = 4$ (this ensures that $p_{F(1)}\geq p_1+4$).
\item Assume $n>0$ and $\beta^{\alpha}_0, \dots, \beta^{\alpha}_n$ are defined.
Define the finite sequences $(\eta^m_k)_{k\leq n+1}$ by putting $\eta^m_{n+1} = \beta_n - \frac{m}{F^{n+1}(1)-F^n(1)}$ and $\eta^m_k = \beta^{\alpha}_k$ for all $k \le n$, and define the associated bounded density subshifts $\Sigma^{(n)}_m = \Sigma_{p(f,\eta^m})$. 
Using Lemma~\ref{lem:BoundedDecidable}, compute some values $q_n(m) \in \mathbb{Q}$ for all $m \in \llbracket 0 , F^n(1) \rrbracket$, 
such that \[|h(\Sigma^{(n)}_m)-q_n(m)|\leq 2^{-n}.\]
Finally, define $\beta^{\alpha}_{n+1} = \frac{m^\ast_n}{2^n} \beta^{\alpha}_n$, with 
$m^\ast_n$ the smallest $m$ such that: 
\begin{description}
  \item[Entropy condition] $q_n(m) \geq \alpha_n + 2^{-n}$
  \item[Mixing condition] $ p (f, \eta^m)_{F^{n+1} (1)} \geq 
	2 p (f, (\eta^m) )_{F^{n} (1)} + 4$.
\end{description}
\end{itemize}

See an illustration on Figure~\ref{fig.algo.stepn}.\bigskip

\begin{figure}[htp]
\centering
\begin{tikzpicture}[scale=0.4]
\draw[-latex] (-1,0) -- (20,0); 
\draw[-latex] (0,-1) -- (0,16);
\draw[color=red] (0,0) -- (2,4);
\draw[dashed,color=red] (2,4) -- (6,8);
\draw[color=red] (6,8) -- (10,10);
\node at (8,10) {$\beta_n$};
\draw[dashed] (10,10) -- (10,0);
\node at (10,-2) {$F^{n} (1)$};
\draw (10,10) -- (18,14);
\node at (20,14) {$m=2^n$};
\draw[dashed] (18,14) -- (18,0); 
\draw (10,10) -- (18,11); 
\draw (10,10) -- (18,12); 
\draw (10,10) -- (18,13); 
\node at (20,11) {$m=m^{*}_n$};
\node at (18,-2) {$F^{n+1} (1)$};
\node at (1,-2) {$F^{0} (1)$};
\node at (1,3.5) {$\beta_0$};
\node[color=red] at (-2,5) 
{$\varphi^{\alpha}$};
\end{tikzpicture}

\caption{\label{fig.algo.stepn} Illustration of the definition 
of the algorithm. The number $m^{*}_n$ 
is the smallest one such that the mixing 
condition is verified.}
\end{figure}
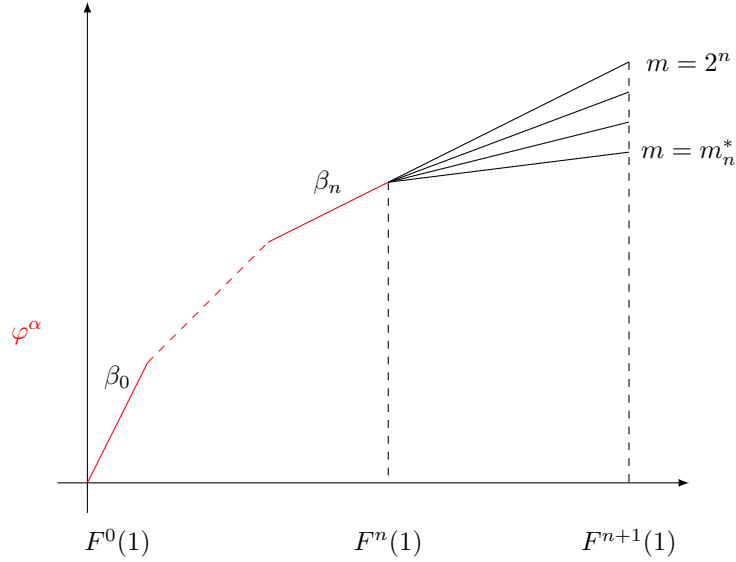

\begin{remark}
A consequence of the entropy condition 
is that for all $n$, $\h(\Sigma_{n,\alpha})\geq 2^{-n}+ \alpha_n$.
\end{remark}

\subsubsection{\label{sec.algo} Proof of 
Theorem~\ref{thm:main} for $d=1$}

In this section, we prove Theorem~\ref{thm:main} for $d=1$ by showing that the subshift we defined in the previous section satisfies the desired properties. Decidability follows from Lemma~\ref{lem:freqDecidable} and $f$-irreducibility from the mixing condition and Lemmas~\ref{lem:abovemixing}. What is left to prove is that its entropy is $\alpha$.

In the following, we use a few shorthands for notation:
\begin{itemize}
\item $\Sigma_{m,\alpha} = \Sigma_{f,(\beta^{\alpha}_k)_{k \le m}}$,
\item $\Sigma_\alpha = \Sigma_{f,\beta^{\alpha}}$,
\item $p^{\alpha} = p(f,\beta^{\alpha})$.
\end{itemize}

\begin{lemma}
\label{lem:aboveentropy}
For any $\Pi_1$-computable real number $\alpha>0$, $\h(\Sigma_{\alpha}) = \alpha$.
\end{lemma}

\begin{proof}[Proof of Lemma~\ref{lem:aboveentropy}]

\begin{description}
\item[Lower bound]

Since ${\mathcal{L}}_n(\Sigma_{\alpha}) = \bigcap_m {\mathcal{L}}_n (\Sigma_{m,\alpha})$, we have that 
\[\h(\Sigma_{\alpha})  = \inf_{n} \frac{\log \#\left({\mathcal{L}}_n(\Sigma_\alpha)\right)}{n} = \inf_{n} \inf_{m} \frac{\log \#\left({\mathcal{L}}_n(\Sigma_{m,\alpha})\right)}{n} = \inf_m \h(\Sigma_{m,\alpha}).\]

Since for all $m$, $\h(\Sigma_{m,\alpha}) \geq \alpha_m + 2^{-m} \geq \alpha$ by the entropy condition of the algorithm, 
\[\h(\Sigma_{\alpha}) \ge \alpha.\]

\item[Upper bound]

For the sake of contradiction, assume that 
\[\h (\Sigma_{\alpha}) > \alpha.\]
\begin{description}
\item[Behavior of the algorithm]
Since $\h(\Sigma_\alpha)=\inf_n \h(\Sigma_{n,\alpha})$, it follows that $\h(\Sigma_{n,\alpha}) - \alpha \geq H(\frac 1n)+2^{-n}$ for all $n$ large enough.

By Lemma~\ref{lem:boundentropy}, this means that in the definition of $\Sigma_{n,\alpha}$, we could have taken $m^\ast_n-1$ instead of $m_n^\ast$ without breaking the entropy condition. The only other possibility is that taking $m^\ast_n-1$ instead of $m_n^\ast$ would have broken the mixing condition. In other words (using the notation from the algorithm):
\begin{equation}
  p(f,\eta^{m_n^\ast-1})_{F^{n+1}(1)} < 2p(f,\eta^{m_n^\ast-1})_{F^{n}(1)}+4.
  \label{eq:mixing}
  \end{equation}
Since $\eta^{m_n^\ast-1}_{n+1} = \eta^{m_n^\ast}_{n+1} - \frac 1{F^{n+1}(1)-F^n(1)}$ and $\eta^{m_n^\ast-1}_k = \eta^{m_n^\ast}_k$ for $k\leq n$, it follows that:
\begin{itemize}
  \item $p(f,\eta^{m_n^\ast-1})_{F^{n}(1)} = p(f,\eta^{m_n^\ast})_{F^{n}(1)}$
  \item $p(f,\eta^{m_n^\ast-1})_{F^{n+1}(1)} = p(f,\eta^{m_n^\ast})_{F^{n+1}(1)}-1$
\end{itemize}
Using these equalities in Equation~\ref{eq:mixing}, we obtain:
\[p(f,\eta^{m_n^\ast})_{F^{n+1}(1)} \leq 2p(f,\eta^{m_n^\ast})_{F^{n}(1)}+3.\]

By definition of $p^\alpha$, this means that for $n$ large enough, 
\[p^\alpha_{F^{n+1}(1)} \leq 2p^\alpha_{F^{n}(1)}+3.\]

\item[Contradiction using the limit density]

By the previous equation, there is a constant $N$ such that for all $n\geq N$, 
\begin{align*}
\frac{p^\alpha_{F^{n+1}(1)}}{F^{n+1}(1)} &\leq \frac{2p^\alpha_{F^{n}(1)}+3}{F^{n+1}(1)}\\
&\leq \frac{p^\alpha_{F^{n}(1)}}{F^{n}(1)}\cdot \frac{2F^n(1)}{F^{n+1}(1)} + \frac{3}{F^{n+1}(1)}\\
&\leq \frac{p^\alpha_{F^{n}(1)}}{F^{n}(1)}\cdot \left(1-\frac{f(F^n(1))}{F^{n+1}(1)}\right) + \frac{3}{F^{n+1}(1)}
\end{align*}

Rewriting this equation,

\[
\frac{p^\alpha_{F^{n}(1)}}{F^{n}(1)} - \frac{p^\alpha_{F^{n+1}(1)}}{F^{n+1}(1)} \geq \left(\inf_{n}
\frac{p_n}{n}\right)\cdot \frac{f(F^n(1))}{F^{n+1}(1)} - \frac{3}{F^{n+1}(1)}
\]

Applying this equation inductively, we get for all $m\geq 0$:
\begin{equation}\label{eq:converge}
\frac{p^\alpha_{F^{n}(1)}}{F^{n}(1)} - \frac{p^\alpha_{F^{n+m}(1)}}{F^{n+m}(1)} \geq \left(\inf_{n}
\frac{p_n}{n}\right)\cdot \sum_{k=n}^{n+m-1}\frac{f(F^{k}(1))}{F^{k+1}(1)} - \sum_{k=n}^{n+m}\frac{3}{F^{k+1}(1)}.
\end{equation}

Since the sequence $(p_n / n)_n$ is bounded, the left-hand side of Equation~(\ref{eq:converge}) is bounded. Since $F(n)\geq 2n$ for all $n$, we have $F^k(1)\geq 2^k$ for all $k$, and $\frac{3}{F^{k+1}(1)} \leq \frac{3}{2^{k+1}}$, so that \[\sum_{k=0}^\infty\frac{3}{F^{k+1}(1)}<\infty,\] and therefore \[\left(\inf_{n}
\frac{p_n}{n}\right)\cdot \sum_{k=n}^{+\infty}\frac{f(F^{k}(1))}{F^{k+1}(1)}<\infty.\]

We prove that $\sum_{k=0}^{m}\frac{f(F^{k}(1))}{F^{k+1}(1)}$ diverges as $m\to\infty$. For all $n$, 
$f(n)\leq (\beta^{\alpha}_0+1) n$. Since $f$ is nondecreasing, we have for all $i$:
\begin{align*}
  \sum_{k=F^{i-1}(1)+1}^{F^i(1)}\frac {f(k)}{k^2}&\leq (F^i(1)-F^{i-1}(1))\cdot\frac{f(F^i(1))}{(F^{i-1}(1))^2}\\& \leq \frac {(F^i(1)-F^{i-1}(1))\cdot F^{i+1}(1)}{(F^{i-1}(1))^2} \cdot \frac{f(F^i(1))}{F^{i+1}(1)}\\ & \leq (
	\beta^{\alpha}_0+2)(
	\beta^{\alpha}_0+3)^2 
	\frac{f(F^i(1))}{F^{i+1}(1)}
\end{align*}
Since $\sum_{k}^\infty\frac{f(k)}{k^2}=+\infty$ by hypothesis, we have $\sum_{i}^\infty \frac {f(F^i(n))}{F^{i+1}(n)}=+\infty$ as well. By considering Equation~\ref{eq:converge} as $m\to\infty$, we see that we must have $\left(\inf_{n}
\frac{p_n}{n}\right) = 0$. By Lemma~\ref{lem:slopetozero}, this implies that $\h (\Sigma_{\alpha})=0$. We have reached a contradiction.
\end{description}

As a consequence, 
\[\h(\Sigma_{p}) = \alpha\]
which is the desired statement.
\end{description}
\end{proof}

The Theorem for $d=1$ follows from Lemma~\ref{lem:freqDecidable} and Lemmas~\ref{lem:abovemixing} and~\ref{lem:aboveentropy}. 

\subsubsection{\label{sec.dimsup} Proof for $d \ge 1$}

In order to obtain the same result in higher dimension, notice that for any one-dimensional subshift $\Sigma$, the subshift
\[\Sigma^d = \{x\in\A^{\Z^d}\ :\ \forall \vec{j} \in\Z^{d-1},\ (x_{i,\vec{j}})_{i\in\Z} \in \Sigma\}\]
has the same entropy and decidability properties as $\Sigma$. We prove that if $\Sigma$ is $f$-irreducible, then $\Sigma^d$ is also $f$-irreducible.\bigskip

Indeed, let $\mathbb{U},\mathbb{V}$ two finite subsets of $\Z^d$ such that $d(\mathbb{U},\mathbb{V}) \ge f(\max(\delta(\mathbb{U}),\delta(\mathbb{V})))$ and $u,v$ two patterns on $\mathbb{U},\mathbb{V}$ 
respectively.
For all $\vec{k} \in \Z^{d-1}$, denote $\mathcal H_{\vec{k}} =\{(i,\vec{k}) : i \in \Z \}$. 
Consider the sets $\mathbb{U}_{\vec{k}} =\mathbb U\cap \mathcal H_k$ and $\mathbb{V}_{\vec{k}} =\mathbb V\cap \mathcal H_k$, and put $u_k = u|_{\mathbb{U}_{\vec{k}}} , v_k=v|_{\mathbb{V}_{\vec{k}}}$.

Since the function $f$ is non-increasing, that $d (\mathbb{U}_{\vec{k}}, \mathbb{V}_{\vec{k}}) \ge \delta (\mathbb{U},\mathbb{V})$ and that furthermore $\max(\delta(\mathbb{U}),\delta(\mathbb{V})) \ge \max(\delta(\mathbb{U}_{\vec{k}}), \delta(\mathbb{V}_{\vec{k}}))$, we have:
\[d  (\mathbb{U}_{\vec{k}},\mathbb{V}_{\vec{k}}) \ge f(\max(\delta(\mathbb{U}_{\vec{k}}), \delta(\mathbb{V}_{\vec{k}}))).\]

By $f$-irreducibility of $\Sigma$, this implies that there exists some $x_k \in \Sigma$ whose restrictions on $\mathbb{U}_{\vec{k}}$ and $\mathbb{V}_{\vec{k}}$ are $u_k$ and $v_k$, respectively.

Let $x$ be the configuration defined by $x|_{\mathcal H_k} =x_{\vec{k}}$ for all $\vec{k} \in \Z^{d-1}$. Then $x\in\Sigma$ and $x|_\mathbb{U} = u$ and $x|_\mathbb{V} = v$ by construction.

\section{Conclusion} 

Our main result is the proof of a jump in the difficulty of computing entropy of decidable subshifts when a measure of mixing strength, the irreducibility rate, passes a certain threshold. We offer some perspectives for further research:
\begin{itemize}
\item We do not have a characterisation of real numbers that can be reached as entropies of decidable subshifts whose irreducibility rate is under the threshold. We conjecture that all computable real numbers can be reached in this way.
\item The main question, and the initial motivation of this work, is whether the same threshold marks the jump between computable and uncomputable entropy for subshifts of finite type of higher dimension. 
\end{itemize}

\section*{Acknoledgments}

The authors are grateful to Ronnie Pavlov for many useful discussions and suggestions.

\bibliographystyle{plain}
\bibliography{Biblio}

\end{document}